\documentclass{amsart}
\usepackage{amsfonts,amscd, amssymb}

\newtheorem{theorem}{Theorem}[section]
\newtheorem{lemma}[theorem]{Lemma}
\newtheorem{corollary}[theorem]{Corollary}
\newtheorem{proposition}[theorem]{Proposition}
\theoremstyle{remark}
\newtheorem{remark}[theorem]{Remark}
\theoremstyle{definition}

\newtheorem{example}[theorem]{Example}
\numberwithin{equation}{section}
\makeatother

\newcommand{\norm}[1]{\Vert#1\Vert}

\DeclareMathOperator{\Cdb}{{\mathbb C}}
\DeclareMathOperator{\Rdb}{{\mathbb R}}
\DeclareMathOperator{\Zdb}{{\mathbb Z}}

\DeclareMathOperator{\Ddb}{{\mathbb D}}
\DeclareMathOperator{\Tdb}{{\mathbb T}}
\DeclareMathOperator{\Ndb}{{\mathbb N}}

\begin{document}

\title[Real operator algebras and real positive maps]{Real operator algebras and real positive maps}
\thanks{Supported by a Simons Foundation Collaboration Grant.}
\author{David P. Blecher}
\address{Department of Mathematics, University of Houston, Houston, TX
77204-3008, USA}
\email[David P. Blecher]{dpblecher@central.uh.edu}

\author{Worawit Tepsan}
\address{Chiang Mai University, 
239 Nimmanhaemin Road, Suthep, Muang, 
Chiang Mai, Thailand, 50200}
\email[Worawit Tepsan]{worawit.tepsan@cmu.ac.th}

%\date{6/30/2020} 

\subjclass{Primary 17C65 , 46L05,  46L70, 47L05, 47L07, 47L30, 47L70; Secondary: 46H10, 46B40,  46L07, 46L30, 47B92, 47L75}
%\subjclass{Primary  47A64, 47L10, 47L30, 47B44; Secondary   15A24, 15A60, 47A12, 47A60, 47A63, 49M15, 65F30}
\keywords{Operator algebra, real Jordan operator algebra, real operator system, real completely bounded maps, real positive maps, approximate identities}
%Operator algebra, Jordan operator algebra, contractive projection, conditional expectation, real positive, 
%noncommutative Banach-Stone theorem, JC*-algebra}

\begin{abstract}  We  present some foundations for a theory of real 
operator algebras  and real Jordan operator algebras, and the various morphisms
between these.  A common theme is the ingredient of {\em real positivity} from 
papers of the first author with Read, Neal, Wang, and other coauthors,  which we import to the real scalar setting here. 
% For example we find AD1 .   
  \end{abstract}

\maketitle

\section{Introduction}  

In a 2016 paper \cite{Ros}, Rosenberg commented ``Over the last thirty years ... it has become apparent that real $C^*$-algebras have a lot of extra structure not evident from their complexifications. At the same time, interest in real $C^*$-algebras has been driven by a number of compelling applications." He goes on to mention 
the classification of manifolds of positive scalar curvature,  representation theory,  the study of orientifold string theories, and the real Baum-Connes conjecture.
%We mention that e
Even the problem of whether a complex $C^*$-algebra is the complexification of a real $C^*$-algebra is a difficult one, with contributions 
by Connes \cite{C}, Jones, St{\o}rmer, etc.
%and many others.  

On the other hand, 
B\"ottcher and Pietsch  comment in 
%their 2012 paper 
\cite{BP}  that the impression sometimes given that results about complex 
Hilbert spaces carry over mutatis mutandis to the real case,  
should not be taken too literally.  They go on to point out that 
searching in 
Mathematical Reviews for publications whose title contains the term ‘Hilbert space’ yields an output of approximately 10000, while asking for ‘real Hilbert space’ reduces the result to
100 (we have changed their numbers to the current ones).   They suggest that this indicator, while 
%admittedly  
very rough, nonetheless  
displays the shameful treatment of the real case. 
Similarly there has been
 very little work done on real operator algebras relatively speaking, and there is  a frequent 
 lack of appreciation of the considerable difficulties that can arise.     
 
In the present paper we present some foundations for a theory of real 
operator algebras  and real Jordan operator algebras, and the various morphisms
between these (see also \cite{Sharma,WTT}).  A common theme is the ingredient of {\em real positivity} from 
papers of the first author with Read 
in \cite{BRI,BRII,BRord} (see also e.g.\ \cite{BBS,BSan,BNp,BWj,BNj,BNjp,Bposx})
%z
in the complex case, which we import to the real scalar setting here.    
A bounded linear operator on a Hilbert space $H$ is real positive  if $T + T^* \geq 0$.  Real positivity
is intended to be, for nonselfadjoint operator algebras or unital operator spaces,  a substitute for positivity in $C^*$-algebras.  
One indication that real positivity will be useful in the real theory  is the fact (which follows from e.g.\ Lemma \ref{sfun}), that states on a real 
$C^*$-algebra are the norm 1 real positive functionals, but they are not the norm 1 positive functionals (see the
example above Proposition 4.1 in \cite{ROnr}).  

An associative real (resp.\ complex) {\em operator algebra} is a possibly nonselfadjoint closed real (resp.\ complex) subalgebra of $B(H)$, for a
real (resp.\ complex) Hilbert space $H$.   See e.g.\ \cite{BLM} for the complex case of this theory, and
\cite{Sharma} for a preliminary study of the real case.  By a real (resp.\ complex) {\em  Jordan operator algebra} we
 mean a  norm-closed  real (resp.\ complex) {\em  Jordan subalgebra} $A$ of a real (resp.\ complex) $C^*$-algebra,  
namely a norm-closed   real (resp.\ complex) subspace closed under the 
`Jordan product' $a \circ b = \frac{1}{2}(ab+ba)$. Or equivalently,
 with $a^2 \in A$ for all $a \in A$ (this follows since $a \circ b = \frac{1}{2} ((a+b)^2 -a^2 -b^2)$).    A characterization of these algebras is given in 
\cite[Section 4.3]{WTT} and in Theorem 2.1 in \cite{BNj,BWj}.    
The selfadjoint case, that is,  closed selfadjoint real (resp.\ complex) subspaces of a 
 real (resp.\ complex) $C^*$-algebra which are closed under squares, 
will be called real (resp.\ complex) {\em $JC^*$-algebras}. 
Complex  $JC^*$-algebras have a large literature, see e.g.\ \cite{Rod, HS,Top} for references.     
The theory of (possibly nonselfadjoint) Jordan operator   algebras over the complex field was initiated in \cite{BWj,BNj} 
 (see also \cite{BWj2,ZWthes}  for some additional results, complements, etc).   

Of course every operator algebra is a Jordan operator algebra.
Jordan operator algebras are the most general setting for many of the results below.
Thus we state such results in this setting; 
the reader who does not care about nonassociative algebras should simply
restrict to the associative case.    Indeed the statement of some of our results 
contain the phrase `(Jordan) operator algebra', this invites the reader to simply ignore the word `Jordan'. 
Another reason to consider  real Jordan operator algebras is that 
there are many examples, for example in quantum physics.     In addition to the complex examples there are the {\em real $JC$-algebras}
(cf.\ \cite{HS,Top}), namely Jordan algebras of selfadjoint operators on a real Hilbert space.   
We mention for example spin systems: families of selfadjoint unitaries $\{ u_i : i \in I \}$ on 
a real Hilbert space $H$ with   
$u_i \circ u_j = 0$ if $i \neq j$.   For example the Pauli matrices, and the usual appropriate tensor product of these, are good examples of these.
The closed real span of such a family, and the identity $I$, is a real unital Jordan operator subalgebra  of $B(H)$, which
is isomorphic to a Hilbert space (see e.g.\ p.\ 175 in \cite{Pisbk}).   

We are able to include a rather large number of results 
in a relatively short manuscript since many proofs are similar to their
complex counterparts, and thus we often need only discuss the 
new points that arise.    Note that real (Jordan) operator algebras can be treated 
to some extent in the framework of complex (Jordan) operator algebras with an involution
\cite{BWinv}.   Indeed the study of these real algebras is 
`the same as'  (in some sense--but do not take this too literally) the study of 
complex (Jordan) operator algebras with a period 2 conjugate linear (Jordan) automorphism,
which is isometric or completely isometric.  Indeed the fixed points of such an automorphism is a
real (Jordan) operator algebra, and conversely the complexification of a real (Jordan) operator algebra
has such an automorphism $x + iy \mapsto x - iy$.   Similarly for dual real (Jordan) operator algebras,
except now the automorphism is weak* continuous.     

 Thus often a result about real spaces is proved by applying the complex variant of the result to the complexification.
 We will rarely spend much time on such results, focusing rather on situations where complexification fails. 
  In the latter case one may try to copy the proof of the complex variant, or invent a completely new
argument, or attempt a mixture of these two.   Indeed there are a few pitfalls that one must
beware of, and standard tricks in the complex case that fail for real spaces.   For example, for an operator $T$ on a real Hilbert space $H$
the condition $\langle T\xi,\xi \rangle \geq 0$ for every $\xi \in H$ does not imply that $T \geq 0$.  Indeed
arguments involving states or numerical range often do not work in the real case.   E.g.\ the matrix  $E_{12} - E_{21}$  
takes value $0$ at every real state on $M_2(\Rdb)$, but it is not Hermitian (selfadjoint).     We know of no 
simple test for $T \in B(H)$ to be positive or selfadjoint in the real case.   Thus some aspects and directions in the theory
do not generalize very well to the real case, as we shall see e.g.\ in parts  of Section 5. 

%In the present paper  w
We do not attempt to be exhaustive, rather we are content to prove enough results here so that  the reader who is familiar with 
the complex theory from the literature listed in the third paragraph 
will feel, after finishing our paper, that they have a good grasp of how the real case works out.  

In Section 2 we give some basic results on real completely positive maps on  real operator algebras and real Jordan operator algebras, and  completely positive maps on  real operator systems.  
In Section 3 we establish the variant of Meyer's unitization theorem for contractive  homomorphisms on a 
real Jordan operator algebra. Namely for a real Jordan subalgebra $A$ of $B(H)$ not containing $I_H$, 
any  contractive  Jordan homomorphism from $A$ to $B(K)$ 
extends to a unital contractive homomorphism from $A + \Rdb I_H$ to $B(K)$.   This result does not just follow by complexification--we have to exploit 
the Cayley transform as in Meyer's remarkable  original proof.  
This implies in particular that the unitization of a real Jordan operator algebra is uniquely defined up to 
 isometric algebra  isomorphism.   
 In Section 4 we study contractive approximate identities (or {\em cais}), generalizing  to the real case some of the main results 
 from \cite{BWj} (and from the sequence of papers mentioned above) about  cais and about 
 {\em approximately unital} algebras.    
 
 In Section 5 we study real  positivity and real  positive maps, generalizing  to the real case some of the main results 
 from \cite[Section 2]{BNjp}.  
 Thus Section 5 is the appropriate variant for real Jordan operator algebras and real unital operator spaces, of 
 the theory of positive (but not completely positive maps) on complex $C^*$-algebras or operator systems.  
  It turns out that if we are interested in real  positive maps as opposed to RCP maps (see Section 2)
 then an extra condition is usually needed for the real theory, namely the notion of {\em systematic real  positive}.
 However even then some obstacles emerge that do not exist in the complex case.  We give several examples  of bad behavior.
 For example unital contractions on a real operator system need not be selfadjoint, nor positive, and need not extend
 to a contraction on the unitization.   
 Also real positive maps (resp.\  positive selfadjoint maps) need not extend to a real positive map (resp.\  positive selfadjoint map) on a complexification. 
 We  also list some questions that we do not know the answer to but probably are also evidence of malaise in the real case at this 
 level of generality.    
 
In the remaining part of Section 1 we give some background
and notation.   The underlying scalar field is usually 
$\Rdb$, and all spaces, maps or operators in this paper
are usually $\Rdb$-linear. For background  on completely bounded maps, complex operator spaces, and associative operator algebras, we refer the reader
to \cite{BLM,ER,Pau, Pisbk}.  For complex $C^*$-algebras  the reader could consult e.g.\ 
\cite{P}.     For background  on `selfadjoint' nonassociative and `Jordan type' algebras
see e.g.\ \cite{Rod} and references therein.  For the theory of possibly  nonselfadjoint complex Jordan operator algebras the reader will also want to consult \cite{BWj} frequently
 for background, notation, etc, and will often be referred 
there for various results that are used here.   

The letters $H, K$ are reserved for real Hilbert spaces.
 A {\em projection}  in an algebra 
is always an orthogonal projection (so $p = p^2 = p^*$).    A  (possibly nonassociative) normed algebra $A$  is {\em unital} if it has an identity $1$ of norm $1$, 
and a map $T$ 
is unital if $T(1) = 1$.  We write $X_+$ for the positive operators (in the usual sense) that happen to
belong to $X$.   
A {\em Jordan homomorphism} $T : A \to B$ between Jordan algebras
is  a linear map satisfying $T(ab+ba) = T(a) T(b) + T(b) T(a)$ for $a, b \in A$, or equivalently,
that $T(a^2) = T(a)^2$ for all $a \in A$ (the equivalence follows by applying $T$ to $(a+b)^2$).  We write $M_n(X)$ for the space of $n \times n$ matrices 
over a space $X$.  

Some general background on differences between `real and complex functional 
analysis' may be found in the recent survey \cite{Mos} that was not available at the time of submission of our paper.
 We will assume that the reader is familiar with some basics from the theory of 
real $C^*$-algebras \cite{Li}, in particular that a real $C^*$-algebra $B$ has 
a `unique $C^*$-algebra complexification' such that the map $x + i y \mapsto 
x-iy$ is isometric (here $x,y \in B$).  This unique
complexification may be equivalently characterized
in terms of period 2 conjugate linear $*$-automorphisms on complex $C^*$-algebras. 
We will often use facts from that theory that will be evident for readers
for the complex case, and whose real versions may all be found in \cite{Li}.
 If $X$ is a subspace of a (real or complex) $C^*$-algebra $B$ then we write $C^*(X)$ or $C^*_B(X)$ for the
$C^*$-subalgebra of $B$ generated by $X$.  

Although we shall not use this we mention as a background fact that the first author 
together with Mehrdad Kalantar were 
able to  improve on the characterization of commutative real $W^*$-algebras  from \cite{Li} (see e.g.\ Theorem 6.3.1 
there, and e.g. \cite{IP, ARU2}).     This may be new in this full 
generality (e.g.\ \cite{Li,ARU2} seem to consider separably acting algebras).
%, and again we hope to present it elsewhere.  
Indeed the commutative real $W^*$-algebras  come from the 
commutative complex $W^*$-algebras regarded as real algebras: 

 \begin{theorem}  \label{coisre}   Any  commutative real $W^*$-algebra is (real) $*$-isomorphic 
 to the direct sum of a 
 commutative complex $W^*$-algebra, and the set of selfadjoint elements in a commutative complex $W^*$-algebra.
 This may be written as $L^\infty(X,\mu,\Cdb) \oplus L^\infty(Y,\nu,\Rdb)$.
 Here $(X,\mu)$ and $(Y,\mu)$ are localizable (or Radon) measure spaces 
 (and we allow one of these
 to be the empty set, and in this case interpret the $L^\infty$-space to be $(0)$).
\end{theorem}

Ruan initiated the study of real operator spaces in \cite{ROnr,RComp}, and this study was continued in 
\cite{Sharma}.    A real operator space may either be viewed as a real subspace of $B(H)$, or abstractly as 
a vector space with a norm $\| \cdot \|_n$ on $M_n(X)$ for each $n \in \Ndb$ (satisfying  
{\em Ruan's  axioms}  \cite{ROnr}).    Sometimes the sequence of norms $(\| \cdot \|_n)$ is
called the {\em operator space structure}.  All spaces in the present paper are such operator spaces, although
sometimes we will not care about the higher matrix norms.  
We will say that an operator space complexification $X_c = X + iX$ of an operator space $X$
is {\em completely reasonable} if  the map 
$\iota_X : x+iy \mapsto x - iy$ is a `complete isometry', for $x, y \in X$.  Ruan proved that a real operator space 
$X$ possesses a completely reasonable operator space complexification, which is unique up to complete isometry.
This unique complexification may be identified up to real  complete isometry with the operator subspace of $M_2(X)$ 
%consisting 
of matrices of form 
\begin{equation} \label{ofr} \begin{bmatrix}
       x    & -y \\
       y   & x
    \end{bmatrix}
    \end{equation} 
    for $x, y \in X$.

One may define an {\em operator algebra complexification}
(resp.\  {\em Jordan operator algebra complexification}) of a real operator algebra
(resp.\  Jordan operator algebra) $A$ to be the complex linear span $A + iA$
of $A$ in $B_c$, if $A$ is a subalgebra (resp.\  Jordan subalgebra) of a 
real $C^*$-algebra $B$.   Here $B_c$ is the unique $C^*$-algebra complexification a 
couple of paragraphs above.   
In particular if  $A$ is a real subalgebra (resp.\ Jordan subalgebra)  of $B(H)$ 
 for a real Hilbert space $H$ then 
 the operator algebra (resp.\  Jordan operator algebra) complexification 
 $A + i A$ is a complex subalgebra (resp.\ Jordan subalgebra)  of $B(H)_c = B(H_c)$.
 These definitions have the drawback of being too `concrete', but are easily made 
`more suitable for abstract use' by adding words like `is isometrically 
isomorphic (resp.\  Jordan  isomorphic) to' in a couple of places above.   One uses 
`completely isometrically' here in place of `isometrically'
 if one needs to work in the matrix normed (operator space)
 setting.  
As mentioned much earlier in the introduction, this complexification
can be equivalently defined abstractly in terms
of a period 2 conjugate linear isometric 
(or completely isometric, if one needs to work in the matrix normed setting)
automorphism $\theta$ on a complex operator algebra $B$.
Here $B$ is 
spanned by the subspace of  fixed points of $\theta$, which  we are assuming to be 
an isometric (or completely isometric) `copy' of the  operator algebra (resp.\  Jordan operator algebra) $A$.  We say that such $B$ is an {\em operator algebra (resp.\  Jordan operator algebra) complexification}  of $A$.  
We show that then $B$ is (isometrically isomorphic to) a concrete complexification of 
an  isometric copy of $A$ of the type described 
at the start of this paragraph.  Suppose that $B$ acts  on a Hilbert space $H$, and 
set $D = \{ ( b , \overline{\theta(b)} ) \in B(H) \oplus \overline{B(H)}  : b \in B \}$, a copy of $B$. One may  use the fact that $\overline{B(H)} = B(\bar{H})$ $*$-isomorphically, 
where $\overline{B(H)}$ is the $C^*$-algebra in e.g.\ the remark after Proposition 2.1 in 
\cite{BWinv}.  Define $C$ to be the $C^*$-algebra generated by $D$, and let 
$\rho$ be the period 2 conjugate linear $*$-automorphism on $C$ which is the restriction 
of the $*$-automorphism $(S, \bar{T}) \mapsto (T, \bar{S})$ on  $B(H) \oplus \overline{B(H)}$.
The fixed points of $\rho$ are a real  $C^*$-algebra $M$ with complexification $C$,
and $A$ may be viewed as 
%the subalgebra 
its copy $\{ (a , \overline{\theta(a)} ) : a \in A \}$ in $M$.   Moreover $\rho$ extends 
the copy of $\theta$ on $D$.   Thus the complexification of this copy of $A$ 
as described at the start of this paragraph, is  isometrically isomorphic to $B$. 
The same proof works in the completely isometric case.

 %\begin{theorem}  \label{chj}  Suppose that $\theta : B \to B$ is a  period 2 conjugate linear isometric 
%(resp.\ completely isometric)  automorphism on a complex operator algebra.
%Then the fixed points of $\theta$ are a real Jordan  operator algebra, whose Jordan  operator algebra complexification 
%in the sense above is $B$ isometrically (resp.\ completely isometrically).   Moreover every real Jordan  operator algebra and its complexifications arise in this way. \end{theorem} 

It is important to realize that the operator algebra (resp.\ Jordan operator algebra) complexification $A_c$ is not uniquely defined up to isometric isomorphism or isometric 
Jordan isomorphism.   That is, the  
complexification of a real operator algebra $A$ is not uniquely defined at the Banach algebra level.
 This is clear for example from Proposition \ref{wtex}.
If however we also take the operator space structure of $A$ into account then  $A_c$ is  uniquely defined up to 
completely isometric algebra (resp.\ Jordan  algebra) isomorphism.
This follows for example by \cite[Theorem 2.1]{RComp}.  If $A \subset B(H)$ is a 
real Jordan operator algebra then $A_c$ may be identified up to real  complete isometric isomorphism with 
the Jordan subalgebra of matrices of the form (\ref{ofr}) in  $M_2(A) \cap M_2(B(H))$.
Note that  $A \cap (A_c)^{-1} = A^{-1}$, since e.g.\ if $a(b+ i c) = 1$ then
$iac = 0$ and $ab = 1$.  

If  $A$ is a Jordan operator subalgebra  of $B(H)$, then the {\em diagonal}  
$\Delta(A) = A \cap A^*$  is a $JC^*$-algebra.    An element $q$ in a Jordan operator algebra $A$
 is a projection if $q^2 = q$ and $\| q \| = 1$
(so these are just the orthogonal projections on the 
Hilbert space which $A$ acts on, and which are in $A$).    Clearly $q \in \Delta(A)$. 
A {\em Jordan  contractive approximate identity}
(or {\em J-cai} for short) for a Jordan operator algebra $A$,  is a net
$(e_t)$  of contractions (i.e.\ elements of norm $\leq 1$) with $e_t \circ a \to a$ for all $a \in A$.
If a J-cai for $A$ exists then $A$ is called {\em approximately unital}.    See Lemma \ref{jcai}  for the 
 main result on such approximate identities.   One consequence of this result is that an associative
 operator algebra has a cai if and only if it has a J-cai.  
  
We say that an operator $x$ in $B(H)$ is  {\em real positive}  if $x + x^* \geq 0$.  (Sometimes this is called being {\em accretive}.) 
This is the same as $x$ being  real positive in $B(H_c)$, and hence it follows 
that the characterizations mentioned in  \cite[Lemma 2.4]{BSan} 
of real positive elements in $B(H)$ are still valid.    For example,  $x + x^* \geq 0$ if and only  
$\| I - tx \| \leq 1 + t^2 \| x \|^2$ for all $t > 0$.   If $A$ is a unital subspace or 
(Jordan) subalgebra of $B(H)$ then we define the real positive  elements in
$A$ to be ${\mathfrak r}_{A} = \{ x \in A  : x + x^* \geq 0 \}$.   If $A$ is unital then it follows that the characterizations in  \cite[Lemma 2.4]{BSan} 
are true for $A$, and that the definition of `real positive' or ${\mathfrak r}_{A}$ does not depend on the particular $B(H)$ that $A$ sits in (isometrically and unitally).

For a (Jordan) operator algebra  or unital operator space $A$, because of the uniqueness of unitization up to isometric isomorphism (see Section 3),  
we can define unambiguously ${\mathfrak F}_A = \{ a \in A : \Vert 1 - a \Vert \leq 1 \}$.  Then 
 $$\frac{1}{2} {\mathfrak F}_A = \{ a \in A : \Vert 1 - 2 a \Vert \leq 1 \} \subset {\rm Ball}(A).$$
Note that   $x \in {\mathfrak F}_A$ if and only if 
$x^* x \leq x + x^*$.   This is because $\Vert 1 - a \Vert \leq 1$ if and only if $(1 - a)^* (1 - a) \leq 1$.
It follows that  ${\mathfrak F}_A \subset {\mathfrak r}_A$.

If $T : X \to Y$ we write $T_n$ for the canonical `entrywise' amplification taking $M_n(X)$ to $M_n(Y)$.   
The completely bounded norm is $\| T \|_{\rm cb} = \sup_n \, \| T_n \|$, and $T$ is contractive (resp.\ 
completely  contractive) if  $\| T \|  \leq 1$ (resp.\ $\| T \|_{\rm cb}  \leq 1$). 
A map $T$ is said to be {\em real positive} if it takes  real positive elements to real positive elements.
We say that it is {\em real completely positive} or {\em RCP} if $T_n$ is  real positive for all $n \in \Ndb$.  We remark that one may study some aspects of real positivity or real positive maps within the context of the theory of
numerical range and generalized numerical range (see e.g.\ \cite[Sections 2.1 and 2.9]{Rod}, or  
\cite[Section 4.3]{RodP} for some recent results that may be useful in future studies).
  
If $A$ is a real Jordan subalgebra of a real $C^*$-algebra $B$ then $A^{**}$ with its Arens product  
is a  weak* closed real Jordan subalgebra of the  real von Neumann algebra $B^{**}$. 
This follows by routine techniques as in the complex case (see \cite[Section 4.2]{WTT} 
and \cite[Section 1]{BWj}).   Since the diagonal $\Delta(A^{**})$ is 
a real JW*-algebra  (that is a weak* closed real $JC^*$-algebra), it follows that 
$A^{**}$ is closed under meets and joins of projections.

States on a unital real Jordan operator algebra $A$ (that is, unital contractive 
functionals) extend to states on any Jordan operator algebra complexification 
$A_c$ by the real Hahn-Banach theorem.    We will discuss states on approximately unital algebras in Section \ref{ai}.

\section{Real completely positive maps on  real operator algebras} \label{Sec2}

\begin{lemma} Let $A$ be a real $C^*$-algebra and let $A_c$ be its complexification.  If $x,y \in A$ then $x+iy\geq 0$ in $A_c$ if and only if $\begin{bmatrix}
       x    & -y \\
       y   & x
    \end{bmatrix}$ is positive in $M_2(A).$
\end{lemma}

\begin{proof}   For any complex $C^*$-algebra $B$ the map taking the $C^*$-subalgebra of $M_2(B)$ consisting of $2 \times 2$ matrices of the above form,
for $a, b \in B$, to $(a+ib,a-ib)$, is
a faithful $*$-homomorphism 
onto $B \oplus B$.   Now set $B = A_c$.  The matrix in the lemma is positive in $M_2(A)$ if and only if it
 is positive in $M_2(A)_c = M_2(A_c)$, hence  if and only if $(x+iy , x-iy)$  is positive in $B \oplus B$.  
Also the map $x+iy \mapsto x-iy$ is a $*$-automorphism of $B= A_c$, hence is positive.
Thus the matrix in the lemma is positive if and only if  both $x \pm i y \geq 0$, and if and only if  $x + iy \geq 0$.  
\end{proof}

Another proof of the last result may be found in \cite[Lemma 3.3.1]{WTT}.  This uses the spectrum and the fact, which is easy to see, that if $x,y \in A$ then $x+iy$ is selfadjoint in $A_c$ if and only if $x$ is selfadjoint and $y$ is antisymmetric (i.e.\ $y^* = -y$).  

 Let  $X \subset B(H)$ be an operator space which is selfadjoint (that is
 $x^* \in X$ if $x \in X$).
Write $X_{\rm sa}$ (resp.\ $X_{\rm as}$) for the selfadjoint  (resp.\ antisymmetric) elements in $X$.
 Note that $X = X_{\rm sa} \oplus X_{\rm as}$ (using the relation $x = \frac{1}{2} (x+ x^*) + \frac{1}{2} (x- x^*)$).  

\begin{lemma}   \label{Tsyjo}  If $X, Y$ are real selfadjoint operator spaces and $T : X \to Y$  is real linear then $T$ is selfadjoint if and only if 
$T(X_{\rm sa}) \subset Y_{\rm sa}$ and $T( X_{\rm as}) \subset Y_{\rm as}$.  \end{lemma} 

\begin{proof}    The one direction is obvious.  If $T(X_{\rm sa}) \subset Y_{\rm sa}$ and $T( X_{\rm as}) \subset Y_{\rm as}$
and $x = \frac{1}{2} (x+ x^*) + \frac{1}{2} (x- x^*)$, then $T(\frac{1}{2} (x+ x^*)) \in  Y_{\rm sa}$ and
$T(\frac{1}{2} (x- x^*)) \in  Y_{\rm as}$.  Thus 
$$T(x)^* = T(\frac{1}{2} (x+ x^*)) – T( \frac{1}{2} (x- x^*)) = T(x^*)$$
as desired.   \end{proof}

A {\em real operator system}  is a selfadjoint unital real subspace $X$ of $B(H)$ for a real Hilbert space $H$ (or of  a real unital $C^*$-algebra). 
We say that $x \in X$ is positive if and only if $x$ is positive in $B(H)$; so $X_+ = X \cap B(H)_+$.
One usually considers operator systems together with their matrix structure, with morphisms the completely positive selfadjoint maps,
or sometimes the completely positive unital selfadjoint maps.   The matrix  structure 
consists  usually of the positive cones $M_n(X)_+ \subset M_n(B(H))_+ = B(H^{(n)})_+$, for all
$n \in \Ndb$.   A {\em real unital operator space}  is a unital subspace $X$ of a real unital $C^*$-algebra (or of $B(H)$).
Again one usually considers these  together with their operator space (i.e.\ matrix norm) structure, with morphisms the  unital completely contractive maps.
 
A unital selfadjoint map defined on a real operator system $X$ is completely contractive if and only if it is completely positive, by \cite[Proposition 4.1]{ROnr}.  
From this it is easily seen that the positive cone $X_+$ (and $M_n(X)_+$ for all $n \in \Ndb$)  is independent of a choice of representation.

If $T : X \to Y$ then $T_c:X_c\to Y_c$ is defined as $T_c(x+iy)=T(x)+iT(y)$ for
$x,y \in A$.

\begin{lemma} \label{lemos} Let $X$ be a real operator system and $T:X\to B(H)$ be a completely positive map. Then $T_c:X_c\to B(H)_c$ is completely positive.
Also  $T$ and $T_c$ are selfadjoint
and  $\| T \|  = \| T \|_{\rm cb} = \| T (1) \|$. 
\end{lemma}
\begin{proof}  We will temporarily write $\iota$ for $i \in \Cdb$, to avoid confusion with usual matrix subscripting.
Let $w = [x_{ij}+\iota \, y_{ij}]\geq 0$ in $M_n(X_c)=M_n(X)_c$, and set $z = \begin{bmatrix}
       x_{ij}    & -y_{ij}\\
       y_{ij}   & x_{ij}
    \end{bmatrix}$.  
Then by the Lemma above, $z \geq 0$ in $M_{2n}(X)$. Since $T$ is completely positive, then 
    $$T_{2n}(z) = \begin{bmatrix}
       T(x_{ij})    & -T(y_{ij})\\
       T(y_{ij})   & T(x_{ij})
    \end{bmatrix}\geq 0$$ in $M_{2n}(X)$. Thus by the same lemma above, $(T_c)_n([x_{ij} +\iota \, y_{ij}]) = [[T(x_{ij})+\iota \, T(y_{ij})]\geq 0$.
    Thus $T_c$ is completely positive.  
    Since a positive map between complex operator system is selfadjoint, $T_c$ is selfadjoint. Thus $T$ is also selfadjoint.   The norm and  completely bounded norm of  $T_c$ equal $\| T (1) \|$ by e.g.\ \cite[Proposition 3.6]{Pau}, hence 
 $\| T \|  = \| T \|_{\rm cb} = \| T (1) \|$. 
\end{proof}

\begin{proposition} \label{rcps} Let $X$ be a real unital operator space or approximately unital real Jordan operator algebra  
and $T:X\to B(H)$ real completely positive. Then $T_c$ is real completely positive.
Also $T$ has a well defined 
extension $\tilde{T}$ to $X + X^*$ which is selfadjoint and completely positive.
Also  $\| T \|  = \| T \|_{\rm cb} = \| \tilde{T} \|_{\rm cb}$.
This equals $\| T (1) \|$ in the unital case, and in the  nonunital case is
$\sup_t \, \| T(e_t) \|$ ($= \lim_t \, \| T(e_t) \|$) for any J-cai $(e_t)$ for $A$. 
 \end{proposition}
 
\begin{proof} We follow  the idea and notation of Lemma \ref{lemos}. Let $w = [x_{ij}+\iota \, y_{ij}]$.
If $w + w^* \geq 0$ then $$\begin{bmatrix}
       x_{ij}    & -y_{ij}\\
        y_{ij}   & x_{ij}
    \end{bmatrix}+\begin{bmatrix}
       x^*_{ji}    & y^*_{ji}\\
        -y^*_{ji}   & x^*_{ji}
    \end{bmatrix}  =  \begin{bmatrix}
       x_{ij}    & -y_{ij}\\
        y_{ij}   & x_{ij}
    \end{bmatrix}+ \begin{bmatrix}
       x_{ij}    & -y_{ij}\\
        y_{ij}   & x_{ij}
    \end{bmatrix}^* \geq 0$$ in $M_{2n}(X)$. Since $T$ is completely real positive, $$\begin{bmatrix}
       T(x_{ij})    & -T(y_{ij})\\
        T(y_{ij})   & T(x_{ij})
    \end{bmatrix}+\begin{bmatrix}
       T(x_{ij})    & -T(y_{ij})\\
        T(y_{ij})   & T(x_{ij})
    \end{bmatrix}^* \geq 0 .$$ 
   Reversing the steps at the start of the proof, but applied to the last matrix,  we see that
    $[T(x_{ij})+ \iota \, T(y_{ij})]+[T(x_{ij}) + \iota \, T(y_{ij})]^* \geq 0$.  Hence $T_c$ is completely real positive.
    
  In the unital case, by the complex theory (see e.g.\ \cite[Theorem 2.5]{BBS}) $T_c$ is bounded and extends to a selfadjoint
 completely positive map $X_c + (X_c)^* \to B(H_c)$.   
 % (Here  $B = B(H)$, but we write the computation so as to work for  more general $B$.)
This restricts to a  selfadjoint completely positive map $X + X^* \to B(H)$.    
 By Lemma \ref{lemos} we have $\| \tilde{T} \|_{\rm cb} = \| T (1) \| = \| \tilde{T} \|$, which implies 
 that   the norm and  completely bounded norm of  $T$ equals   $\| T (1) \|$   too.  
 
In the approximately unital real Jordan operator algebra case,  by the complex theory (see  
 \cite[Lemma 2.1]{BNjp}  and \cite[Proposition 4.9]{BWj}) $T_c$ is bounded, so  $T$ is bounded, 
and $\| T_c \| = \| T \| = \|  T_c^{**}(1) \| = \| T^{**}(1) \|$.    Alternatively, let $\widehat{T_c} : X_c^{**} \to B(H_c)$ be the 
canonical weak* continuous extension.   This is RCP, e.g.\ by an argument in the proof of 
\cite[Theorem 2.6]{BBS}.   The restriction to $W = X^{**}$ is  the 
canonical weak* continuous extension $\hat{T} : X^{**}  \to B(H)$.    Then $$\|  \widehat{T_c}  \|_{\rm cb} = 
 \|  \widehat{T_c}  \| = \|  \widehat{T_c} (1) \| = \|  \hat{T} (1) \|.$$
 We may then apply the unital case to $T^{**}$ (resp.\  $\hat{T}$).    
For example the canonical extension of this weak* continuous map to $W + W^*$ is (selfadjoint and) completely positive. 
Restricting to $X + X^*$ we see that the same is true for the canonical extension $\tilde{T}$ to $X + X^*$.  
We have $$\| T \|  \leq \| \tilde{T} \| \leq \| \tilde{T} \|_{\rm cb} \leq  \|  \widehat{T_c}  \|_{\rm cb}  
=  \|  \hat{T} (1) \|  \leq \sup_t \, \| T(e_t) \| \leq  \| T \| ,$$  for any J-cai $(e_t)$ for $A$. 
Thus  these are all equal.    A similar argument works with $\hat{T}$ replaced by $T^{**}$.
We show that this number equals $\lim_t \, \| T(e_t) \|$.  Indeed if a subnet 
of $( \| T(e_t) \| )$ had limit $< \| T \|$, then by a further replacement we may suppose all terms in this subnet
were bounded above by $\alpha < \| T \|$.  Replacing $(e_t)$ by the appropriate subnet
in the last centered equation would yield the contradiction $\| T \| \leq \alpha < \| T \|$.  
\end{proof} 

\begin{corollary} \label{uccrcp}  If $T : X \to B(H)$ is a unital map on a unital
 real operator space then $T$ is completely contractive if and only if $T$ is real completely positive.  \end{corollary}

\begin{proof} The $(\Leftarrow$)  direction follows from Proposition \ref{rcps}.    The other direction can be seen for example by going to 
the complexification and then using the complex case of the present result (see e.g.\ \cite{BBS}).  
\end{proof} 

By Proposition \ref{rcps} if $X$ is a real unital operator space  then $X + X^*$ is well defined as a real operator system. 
Indeed if $T : X \to Y$ is a surjective unital complete isometry 
between unital operator spaces $X \subset B(H)$ and $Y \subset B(K)$,  then $T$ is real completely positive.
 The canonical extension $\tilde{T}: X + X^*\to B(K): x+y^*\mapsto T(x)+T(y)^*$ is  selfadjoint and is a completely isometric complete order embedding
 onto $Y + Y^*$.   
 As in the proof this may also be seen by extending to the complexification.
 The operator space $X + X^*$ in $B(H)$ has the operator space complexification 
 $(X + X^*)_c \subset B(H_c)$.   In addition, $X_c + X^*_c$ is  a completely reasonable operator space complexification of
 $X + X^*$.  
 By the uniqueness of the operator space complexification (see the introduction), $X_c + X^*_c = (X + X^*)_c.$ 
   
\begin{theorem} \label{contractive-hom} If $\pi:A\to B$ is a homomorphism between real $C^*$-algebras, 
or a Jordan homomorphism between real $JC^*$-algebras, then
$\pi$ is contractive if and only if 
%\item [(ii)] 
$\pi$ is  selfadjoint (hence is a $*$-homomorphism or Jordan $*$-homomorphism).
In this case $\pi$ is positive, indeed completely positive and 
completely contractive in the real $C^*$-algebra case.
%\end{itemize} 
\end{theorem}

\begin{proof} 
Assume that $\pi$ is contractive.  By taking biduals we may assume that $A$ is a real $W^*$-algebra and $B = B(H)$.  
For any projection $p \in A$ we have that  $\pi(p)$ is a contractive idempotent, hence is an orthogonal  projection.   In particular, $\pi(1)$ 
is a projection.  Replacing 
$H$ by $\pi(1) H$, $\pi$ becomes unital.    
 If $x = x^*$ then by the spectral theorem we may approximate $x$ by real linear combinations of projections.
 Using this and  the fact about projections proved at the start of the proof, we see that 
  $\pi(x)$ is selfadjoint.    Suppose that $x^* = -x$ and  that $\varphi$ is a real state on $B$.
 Then $\varphi \circ \pi$ is a real state on $A$, and so  $\varphi (\pi(x)) = 0$.  Thus  $\pi(x)$ is antisymmetric by 
 \cite[Exercise 14A]{Good}
 (see also Lemma 2.1.19  in \cite{WTT}).  Thus $\pi$ is selfadjoint by  Lemma \ref{Tsyjo}.  

If $\pi$ is a Jordan $*$-homomorphism then $\pi_c : A_c\to B_c$ is a Jordan $*$-homomorphism.  
By the corresponding fact for complex Jordan $C^*$-algebras, $\pi_c$ is contractive and positive. Thus, $\pi_{|A}=\pi$ is contractive and positive. 
(The positivity can also be seen more directly by the spectral theorem as in the last paragraph.) 

Let $A$ and $B$ be real $C^*$-algebras. If $\pi:A\to B$ is a contractive homomorphism, then by the above  $\pi$ is a $*$-homomorphism. Then $\pi_c:A_c\to B_c$ is a $*$-homomorphism. Thus, $\pi_c$ is completely positive and  completely contractive by a fact in complex $C^*$-algebras. Since $\pi_c|_A=\pi$, $\pi$ is
completely positive and  completely contractive.  
\end{proof}

The following is an analog of the Stinespring dilation and the Arveson extension theorem for completely positive maps on real unital Jordan operator algebras.

\begin{theorem} \label{Sti} Let $A$ be a unital subspace or real approximately unital Jordan subalgebra of a real  $C^*$-algebra $B$ and let
$T:A\to B(H)$ be a real completely positive map. Then $T$ has a  completely positive extension $\tilde{T}:B\to B(H)$. In addition there is a $*$-representation $\pi:B\to B(K)$ for a real Hilbert space $K$, and a contraction $V\in B(H,K)$, such that $$\tilde{T}(a)=V^*\pi(a)V , \qquad a\in B.$$  Moreover, this can be done with $\norm{T}=\norm{T}_{\rm cb}= \| \tilde{T} \|_{\rm cb}
 =  \norm{V}^2$, and this equals $\norm{T(1)}$ if $A$ is unital. \end{theorem}
 
\begin{proof} By Proposition \ref{rcps} $T$ is completely bounded, with $\| T \|_{\rm cb} = \| T \|$.   If $A$ is unital then by Proposition \ref{rcps}
$T$ has a  unital completely positive extension to $A+ A^*$, and we may extend further by \cite[Proposition 4.2]{ROnr} to a 
 selfadjoint  completely positive map $\tilde{T} : B \to B(H)$, of cb norm $\norm{T(1)}$. 
If $A$ is nonunital let $W = A^{**}$.
 By the proof of Proposition \ref{rcps}
 the canonical extension $u = \widetilde{\hat{T}}$ of $\hat{T}$ to $W + W^*$ is selfadjoint and completely positive, and has the same cb norm.
 We may extend further by \cite[Proposition 4.2]{ROnr} to a 
 selfadjoint  completely positive map $B^{**} \to B(H)$, of cb norm $\norm{\hat{T}(1)}$.    Let $\tilde{T}$ be the restriction to $B$. 
 
 By Theorem 4.3 in \cite{ROnr}, there is a $*$-representation $\pi:B\to B(K)$ where $K$ is a real Hilbert space and bounded operator $V\in B(H,K)$ such that 
$\tilde{T}(a)=V^* \pi(a) V$
for all $a\in B$, and $\norm{V}^2 = \norm{\hat{T}(1)} = \| T \|_{\rm cb} = \| T \|$.  \end{proof}

\begin{corollary} \label{nrp2}   A real positive linear functional on a unital real subspace or  approximately unital real 
Jordan subalgebra of a  real $C^*$-algebra
$B$, extends to a positive selfadjoint functional on $B$ with the same norm.  \end{corollary}

This follows e.g.\ from the last theorem and Lemma \ref{sfun} or can be seen more directly e.g.\ as in the proof of Lemma \ref{sfun}.
Indeed the functionals in the last result are just the positive multiples of {\em states}.

Let $X$ and $Y$ be operator spaces. If $T:X\to Y$ is a completely bounded, then $\| T \|_{\rm cb} = \| T_c \|_{\rm cb}$ by 
\cite[Theorem 2.1]{RComp}.   
However this is not true at the `Banach level', if $T : X\to Y$ is contractive, then $T_c$ may not be contractive. This depends on the operator space structures that are given to $X$ and $Y$, as we shall now see. 

\begin{example}  \label{wl21}  Let $X$ and $Y$ be $l^1_2(\Rdb)$ with the maximal and minimal operator space structures from \cite{Sharma} respectively,  and 
let $T : X \to Y$ be the identity map, a complete contraction.   One obtains a complete contraction $T_c : X_c \to Y_c$.  One can easily show that 
$Y_c$ may be identified completely isometrically with the two dimensional complex $C^*$-algebra 
$l^\infty_2(\Cdb)$ (since $l^1_2(\Rdb) \cong l^\infty_2(\Rdb)$ isometrically).  
On the other hand,  
$X_c$ is $l^1_2(\Cdb)$ with the maximal operator space structure
(which is known to equal its minimal operator space structure).
To see this note that by Propositions 2.6 and 2.3 in \cite{Sharma}, and by the fact above, we have 
that $({\rm Max}(l^1_2(\Rdb)))_c$ equals
$$(({\rm Min}(l^\infty_2(\Rdb)))^*)_c = (({\rm Min}(l^\infty_2(\Rdb)))_c)^* = l^\infty_2(\Cdb)^* = {\rm Max}(l^1_2(\Cdb)).$$ 
We also used the duality of Min and Max for complex operator spaces \cite[Section 1.4]{BLM}. 

Thus $T_c$ cannot be an isometry or complete isometry, since it is well known that 
$l^1_2(\Cdb)$ and $l^\infty_2(\Cdb)$ are not isometrically isomorphic.      Indeed $(T^{-1})_c$ cannot be a contraction, even though $u = T^{-1}$ is an isometry. 
 \end{example} 

\begin{proposition} \label{wtex} There exist real unital operator algebras $A$ and $B$ with operator algebra complexifications  $A_c$ and $B_c$, and a contractive (even isometric) unital homomorphism $\theta : A \to B$ whose complexification $\theta_c: A_c \to B_c$ is not contractive. 
\end{proposition} 

\begin{proof}   Let $X$ and $Y$ be as in Example  \ref{wl21} above.   
We may view $X \subset B(H)$ and let $B = {\mathcal U}(X)$
be the set of `upper triangular' matrices 
$$a=\begin{bmatrix} 
    \alpha  \, I_H     & x \\
    0     & \beta \, I_H
\end{bmatrix}$$ 
where $\alpha,\beta \in \Rdb$ and $x\in X\subseteq B(H)$. 
Note that ${\mathcal U}(X)$ is a real unital operator algebra, and is a subspace of the 
real Paulsen system ${\mathcal S}(X) = 
{\mathcal U}(X) + {\mathcal U}(X)^*$ (see \cite[Lemma 4.12]{Sharma} and the lines above it).     
We claim that 
${\mathcal U}(X)_c = {\mathcal U}(X_c)$ and ${\mathcal S}(X)_c = {\mathcal S}(X_c)$
(these complexifications are the unique operator space complexifications). 
Indeed this claim  follow easily from the facts that 
$M_2(B(H))_c = M_2(B(H)_c) = M_2(B(H_c))$, 
%and ${\mathcal S}(X) \subset {\mathcal S}(X)_c \subset M_2(B(H))_c$, 
and ${\mathcal S}(X) \subset {\mathcal S}(X_c) \subset M_2(B(H)_c)$.  
Following the proof of Proposition 2.2.11 in \cite{BLM}, we obtain that 
\begin{equation}\label{eq3}
\|a\|^2=\sup\{(|\alpha|\sqrt{1-t^2}+\|x\|t)^2+|\beta t|^2 : t\in[0,1] \}.\end{equation}
From this equation, we can easily see that 
$$\bigg\| \begin{bmatrix}
    \alpha       & x \\
        0     &  \beta
\end{bmatrix}\bigg\|=
\bigg\|\begin{bmatrix}
   | \alpha |        & \|x\| \\ 
        0     & | \beta |
\end{bmatrix}\bigg\|.$$

Similarly 
$A = {\mathcal U}(Y)$ is an operator algebra.   By the last norm formula the isometry $u : Y \to X$ in Example  \ref{wl21}  extends to an  isometric unital
homomorphism $\theta_u : A \to B$.   However  suppose that 
$\theta_u$ extended to a contractive unital map $r$ on ${\mathcal U}(Y)_c = {\mathcal U}(Y_c)$.
Then $r$ would be real positive, and hence by e.g.\  the proof of \cite[Lemma 2.1]{BNjp}  it would  extend further
 to a positive selfadjoint  unital map on ${\mathcal S}(Y_c)$.   By e.g.\ (1.25) in \cite{BLM}
 this forces the $1$-$2$-corner map $Y_c \to X_c$ to be contractive.    However this map is $u_c$, giving a contradiction.  
\end{proof} 

Many results in the theory of  complex  operator algebras involving completely contractive maps will be almost identical in the real
case.   For example, Corollary 2.3 and  Corollary 4.18 of \cite{BNj} or  
Theorem 2.5 of \cite{BNp} 
concerning completely contractive projections  $P : A \to A$ on an
operator algebra or Jordan operator algebra, will be true in the real case.
This follows quickly by applying the complex case of these results to $P_c$.
Similarly for Banach-Stone theorems characterizing complete isometries between operator algebras or Jordan operator algebras
(such as \cite[Theorem 3.5]{BNjp} (2) (note $C = B$ there if $B$ is also an operator algebra, by (1)) or 
 \cite[Proposition 6.5]{BNp} or \cite[Theorem 4.5.13]{BLM}).      See e.g.\ \cite[Theorem 4.4]{RComp}.
 
   In passing we mention the 
 Kadison-Banach-Stone theorem for real $JC^*$-algebras (see e.g.\ \cite[Theorem 4.8]{IKR} and \cite{CDRV}): A surjective linear map $T : A \to B$ 
 between real $JC^*$-algebras is 
 an isometry if and only if $T$ is a `triple morphism'  (that is, preserves the natural `triple product').
   In addition if these hold then $T$ is 
 a Jordan homomorphism if and only if it is positive.   We sketch a proof of the last assertion: note that 
 by Theorem \ref{contractive-hom}, a contractive Jordan homomorphism is selfadjoint and positive.   Conversely, if $T$ is 
 a 
 positive isometry then so is $T^{**}$.   This uses the Kaplansky density theorem for real $JC^*$-algebras, which 
 may be proved following a standard proof for the complex case of that result.   Then $u = T^{**}(1)$ is positive and  
 also is a partial isometry, indeed is a unitary in $B^{**}$ in the $JW^*$-algebra sense, since 
 $T^{**}$ is a triple morphism.    
 Hence $u = (u^2)^{\frac{1}{2}} = 1$.  The  triple morphism property  then implies that $T$ is a Jordan homomorphism.
 We do not know if there is a variant of the Banach-Stone theorems above or  in \cite{BNjp} for surjective isometries 
 or surjective real positive isometries between e.g.\ unital real Jordan operator algebras.

Finally we mention some results on the $C^*$-envelope and injective envelope, some of 
benefitted from discussions with Mehrdad Kalantar and which 
we hope to present elsewhere.    There is a difficulty here that we overcome which is
 related to injective envelopes of dynamical systems. 
For the $C^*$-envelope and injective envelope in the complex case we refer to 
 \cite[Chapter 15]{Pau} or \cite[Chapter 4]{BLM}, or the 
papers of Hamana and Ruan referenced there.
 A preliminary study of the injective envelope and $C^*$-envelope in the real case may be found in \cite{Sharma}.   Using notation from those 
sources we are able to prove:

\begin{theorem} \label{ijco}      Let $A$ be  a unital real operator space or operator system, or if $A$ is an approximately unital real  operator algebra (or Jordan operator algebra).  
Then  $I(A)_c = I(A_c)$.   Also,  $I(A)$ is a unital  real $C^*$-subalgebra of $I(A_c)$, 
and if $C^*_e(A)$ is the $C^*$-subalgebra of $I(A)$ generated by $A$
then $C^*_e(A)_c = C^*_e(A_c)$.  
\end{theorem}

In this result, $C^*_e({\mathcal S})$ has the universal  property of the $C^*$-envelope:  given any unital complete isometry 
 $j :  {\mathcal S} \to D$ into a real $C^*$-algebra $D$ such that $j({\mathcal S})$ generates $D$ as a real $C^*$-algebra,
 there exists a $*$-epimorphism $\pi : D \to C^*_e({\mathcal S})$  such that $\pi \circ j$ is the canonical inclusion of ${\mathcal S}$ in 
 $C^*_e({\mathcal S})$.    
 
There is a result  analogous to $I({\mathcal S})_c = I({\mathcal S}_c)$ in the context of Hamana's $G$-injective envelope \cite{Hamiecds,Hamiods}.
Namely that $I_G({\mathcal S}) = I({\mathcal S})$ for a finite group
$G$ and an operator system ${\mathcal S}$ which is a  $G$-module  in the sense of Hamana \cite{Hamiecds,Hamiods}.  
A similar result holds for the $G$-$C^*$-envelope.   The case  of this where $G = \Zdb_2$ was the inspiration
for the last proof.    We hope to present this elsewhere in work with Mehrdad Kalantar and a graduate student.

\section{Unitization (Meyer's theorem)} \label{uanf} 

In \cite[Theorem 3.5]{Sharma} a real variant of Meyer's unitization theorem was proved for completely  contractive homomorphisms
on real operator algebras.
Namely any completely  contractive real linear homomorphism $A \to B(K)$ on a subalgebra $A$ of $B(H)$ not containing $I_H$,
extends to a unital  completely  contractive real linear homomorphism $A + \Rdb I_H \to B(K)$.   
This implies that the unitization of a real operator algebra is uniquely defined up to 
completely isometric algebra  isomorphism.   
The variant of Meyer's theorem for contractive $\Cdb$-linear homomorphisms on complex Jordan operator algebras
was noted in \cite{BWj}.    However it is more difficult to prove the real version of the latter result, and we turn to this next.

\begin{lemma}\label{aboveMeyer1} Let $A\subseteq B(H)$ be a real (Jordan) operator algebra 
and $A_c\subseteq B(H_c)$ be its complexification where $H$ is a real Hilbert space. Assume that $I_{H}\notin A$. 
Then for $a,b\in A$ and $\lambda\in \Cdb$, we have 
$$|\lambda|\leq \|(a+ib)+\lambda I_H\|.$$
\end{lemma}
\begin{proof}  We may replace $A_c$ by the closed algebra generated by $a+ib$.   Then 
this follows from  \cite[Lemma 2.1.12]{BLM}. \end{proof}

\begin{theorem}[Meyer type unitization]\label{Meyer-Real-Unique} Let $A$ be a real subalgebra (resp.\ Jordan subalgebra) 
of $B(H)$, and assume that $I_H \notin A$. Let $\pi: A\to B(K)$ be a contractive homomorphism 
(resp.\ Jordan homomorphism) for a real Hilbert space $K$. Let $A^1= span_{\Rdb}\{A, I_H\} \subseteq B(H)$ and define $\pi^o: A^1\to B(K)$ by $\pi^o(a+\lambda I_H)=\pi(a)+\lambda I_K$. Then $\pi^0$ is a contractive homomorphism (resp.\ contractive Jordan homomorphism).
\end{theorem}

\begin{proof} We follow the proof of Meyer's theorem for a complex operator algebra (see Theorem 2.1.13 in  \cite{BLM}) using the fact that $A$ has a complexification which is a complex operator algebra.     

It is easy to see that $\pi^0$ is a  homomorphism (resp.\ Jordan homomorphism).   To show that it is contractive,  
let $T=a+\lambda I_H \in A^1$ for some $a\in A$ and $\lambda\in \Rdb$ and $\|T\|< 1$.  We may effectively replace $A$ 
by the closed algebra generated by $a$, which is an operator algebra.
   We claim that $\|\pi^\circ(T)\|< 1$. 
We will regard everything as objects inside $B(H_c)$.   In particular we view $A, A^1 = A + \Rdb I_H$, and $B(H)$ as 
 real subalgebras of 
$B(H_c)$, and we view 
$T$ as an operator in $B(H_c)$.   By Lemma \ref{aboveMeyer1}, $|\lambda|<1$.  Since $T$ is strictly contractive, by item (2) 
in 2.1.14 in \cite{BLM} we have that  $(I+T)(I-T)^{-1}$ is strictly accretive. Set $\alpha=(1+\lambda)/(1-\lambda)$. Then $\alpha>0$ and
$$\theta = \frac{1}{\alpha}(I+T)(I-T)^{-1}  = I+\frac{1}{\alpha}\Big( (I+T)(I-T)^{-1}- (I+\lambda)(I-\lambda)^{-1} \Big)$$
is also strictly accretive. 
%Note that b
By the Neumann lemma, $(I-T)^{-1}=\sum_{k=0}^\infty T^k$, which lies
in $A + \Rdb I$.  We may write
$(I+T)(I-T)^{-1}- (I+\lambda)(I-\lambda)^{-1}$ as 
$$ (I-T)^{-1} \Big(  (I+T)(I-\lambda)-(I-T)(I+\lambda)   \Big)(I-\lambda)^{-1} , $$
which equals 
%\begin{equation*}
%\begin{split} 
% (I-T)^{-1} \Big(  (I+T)(I-\lambda)-(I-T)(I+\lambda)   \Big)(I-\lambda)^{-1} 
%&= 
$2(I-T)^{-1} a (I-\lambda)^{-1}
=\frac{2}{1-\lambda}(I-T)^{-1} a.$
%\end{split} \end{equation*} x
%\begin{equation*} \begin{split} 
% (I-T)^{-1} \Big(  (I+T)(I-\lambda)-(I-T)(I+\lambda)   \Big)(I-\lambda)^{-1}  &= 2(I-T)^{-1} a (I-\lambda)^{-1}\\
%&=\frac{2}{1-\lambda}(I-T)^{-1} a. \end{split} \end{equation*}
Since $A$ is an ideal in  $A + \Rdb I$, we have
$$\theta-I=\alpha^{-1}(I+T)(I-T)^{-1}- (I+\lambda)(I-\lambda)^{-1}\in A .$$
Since $\theta$ is accretive,  $\theta+I$ is invertible.
By the principle $A \cap (A_c)^{-1} = A^{-1}$ mentioned in the introduction, we have
$(\theta+I)^{-1}\in A + \Rdb 1$. So  we must have $(\theta-I)(\theta+I)^{-1}\in A$,
again since $A$ is an ideal of $A + \Rdb 1$.  

 Since $\pi_c^o$ is a unital homomorphism and $\theta+I$ is invertible, $\pi_c^o(\theta+I)=\pi_c^o(\theta)+I$ is invertible and $\pi_c((\theta+I)^{-1})= (\pi_c^o(\theta)+I))^{-1}.$ Thus,
$$\pi_c^o((\theta-I)(\theta+I)^{-1})=(\pi_c^0(\theta)-I)(\pi_c^o(\theta)+I)^{-1}.$$  We will use items (1) and  (2) 
in 2.1.14 in \cite{BLM} several times.  
We know that $\theta$ is strictly accretive, thus $(\theta-I)(\theta+I)^{-1}$ is strictly contractive and is an element in $A \subseteq A_c$. Since $\pi_c^o|_A=\pi$, $\pi^\circ_c((\theta-I)(\theta+I)^{-1})\in B(K).$    Since $\pi$ is a contraction, 
$\| (\pi_c^0(\theta)-I)(\pi_c^o(\theta)+I)^{-1}\|_{B(K_c)}$ equals
$$\| \pi_c^\circ((\theta-I)(\theta+I)^{-1})\|_{B(K_c)}= \|\pi ((\theta-I)(\theta+I)^{-1})\|_{B(K)} < 1.$$
Thus, $\pi_c^o(\theta)$ is strictly accretive in $B(K_c)$. Thus,
$$\alpha \pi_c^0(\theta)=\pi_c^o((I+T)(I-T)^{-1})=(I+\pi_c^o(T))(I-\pi_c^o(T))^{-1}$$
is strictly accretive. Therefore $\pi_c^o(T) = \pi^\circ(T)$ is strictly contractive as desired.
\end{proof}

It now follows that the unitization of a real operator algebra is unique up to isometric isomorphism: 

\begin{theorem} \label{Meyer-Real-Algebra} 
Let $A$ be a real subalgebra (resp.\ Jordan subalgebra) 
of $B(H)$, and assume that $I_H \notin A$.  Let $\pi: A\to B(K)$ be an isometric homomorphism 
(resp.\ isometric Jordan homomorphism) for a real Hilbert space $K$. 
Then the unital homomorphism $\pi^\circ : A^1\to B(K)$ where $\pi^\circ(a+\lambda I_H)=\pi(a)+\lambda I_K$ is an isometric isomorphism 
onto $\pi(A) + \Rdb \, I_H$.  
\end{theorem}

\begin{proof}  This follows from Theorem \ref{Meyer-Real-Unique} as in \cite[Corollary 2.1.15]{BLM}.  \end{proof}

\begin{corollary} \label{uniqueJoaunitization} The unitization $A^1$ of a  Jordan operator algebra 
is unique up to isometric Jordan isomorphism. In addition, $(A^1)_c=(A_c)^1$ isometrically isomorphically.
\end{corollary}

\begin{proof} We follow the proof of Corollary 2.5 in \cite{BWj}. If $A$ is nonunital  then we may assume that $A$ is represented on a Hilbert space $H$ and 
the Jordan operator algebra unitization $A^1$ of $A$ is identified with $A + \Rdb \, I_H$.   Then 
the first assertion 
follows from Theorem \ref{Meyer-Real-Algebra}.   If $A$ is unital and $e$ is the identity of $A$, then $e$ is a 
central projection of a unitization $A^1$.  Also $e$ commutes with $A^*$ (adjoints with respect to  a fixed unital 
 isometric representation of $A^1$; this follows for example since $e$ is selfadjoint in that representation by the statement 
 about $\Delta(A)$ in Lemma \ref{ApA}).
 So  $e$ is central in $C^*(A^1)$.   From this it is easy to see that for 
$a \in A, \lambda \in \Rdb$ we have  $$\|a+\lambda 1\|=\max\{\|e(a+\lambda 1) \|, \|(1-e)(a+\lambda 1)\|\}=\max\{\|a+\lambda e\|, |\lambda|\}  .$$ 

Since a unitization of complex Jordan operator algebra is  unique up to isometric Jordan homomorphism (see  \cite[Corollary 2.5]{BWj}), $(A^1)_c=(A_c)^1$.
\end{proof}

\section{Approximate identities}  \label{ai}  If $A$ is a Jordan  subalgebra  of a $C^*$-algebra $B$  (either real or complex case) 
 then we say that  a net $(e_t)$ in Ball$(A)$  is 
a $B$-{\em relative partial cai} for $A$ if
 $e_t a \to a$ and $a e_t \to a$  for all $a \in A$.  Here we are using the usual
product on $B$, 
which may not give an element in $A$, and may depend on $B$. 
We say that a net $(e_t)$ in Ball$(A)$  is 
a  {\em partial cai} for $A$ if 
for every $C^*$-algebra $B$ containing $A$ as
a Jordan subalgebra,  $e_t a \to a$ and $a e_t \to a$  for all $a \in A$,
using the product 
on $B$.   Note that partial cais are the same as cais if $A$ is an associative operator algebra.
We say that 
$A$ is {\em approximately unital} if it has a partial cai.  
If $A$ is an operator algebra or Jordan operator algebra then we recall that  a net $(e_t)$ in Ball$(A)$  is 
a {\em Jordan cai} or {\em J-cai} for $A$ if  $e_t a + a e_t \to 2a$  for all $a \in A$.  

\begin{lemma} \label{JOA_aprox_unital} Let $A$ be a real Jordan subalgebra of a real $C^*$-algebra $B$. Then 
\begin{enumerate}
    \item $A$ has a $B$-relative partial cai if and only if $A_c$ has a $B_c$-relative partial cai.
    \item $A$ has a J-cai if and only if $A_c$ has a J-cai.
    \item If $A_c$ has a partial cai then $A$ has a partial cai.
\end{enumerate}
\end{lemma}

\begin{proof}  Note that $A_c$ is a complex Jordan subalgebra of the $C^*$-algebra $B_c$.
Any $B-$relative partial cai of $A$ is a $B_c$-relative partial cai of $A_c$. Conversely, if $(e_t+i \, f_t)$ is a $B_c$-partial cai of $A_c$, then  $(e_t)$ is a $B$-partial cai of $A$.    (See the fact in the 
proof of \cite[Proposition 5.2.4]{Li}.) 
Similarly for J-cai's.   Item (3) follows from the ideas in (1).  
\end{proof}

\begin{lemma} \label{jcai}  If $A$ is a real Jordan subalgebra 
of a real $C^*$-algebra $B$, then the following are equivalent:
\begin{itemize} \item [(i)] $A$ has a partial cai.
 \item [(ii)]  $A$ has a $B$-relative  partial cai. 
\item [(iii)]  
$A$ has a J-cai.
\item [(iv)]   $A^{**}$ has an identity $p$ of norm 1 with respect to  the Jordan Arens product  on $A^{**}$,
which coincides on $A^{**}$ with the restriction of the usual  product
in $B^{**}$.     Indeed $p$  is the identity of the von Neumann algebra
$C^*_B(A)^{**}$.  
\end{itemize}  If these 
hold then  any partial cai $(e_t)$ for $A$  
is a cai for $C^*_B(A)$ (and for the associative operator algebra generated by $A$), and every J-cai for $A$ converges weak* to  $p$.    
\end{lemma}

\begin{proof}   This holds almost exactly as in the complex case \cite[Lemma 2.6]{BWj}.   We just indicate a proof that (iv) implies (i).   Suppose that
$p$ is an identity for $A^{**}$.   Viewing $A$ as with its operator space structure we have that $(A_c)^{**} = (A^{**})_c$ by \cite{RComp}.
Since the canonical map $A_c \to (A^{**})_c$ is a Jordan homomorphism, so is
its weak* continuous extension $(A_c)^{**} \to (A^{**})_c$.   Thus $(A_c)^{**} = (A^{**})_c$ as dual real Jordan operator algebras.
Thus $p$ is the identity of $(A_c)^{**}$.   By the complex case of the present result, 
$A_c$ has a  partial cai $(e_t + i f_t)$, with $e_t, f_t \in A$.  Therefore by the proof of the last lemma $(e_t)$ is a partial cai for $A$.   
\end{proof}

\begin{proposition} \label{coj} Let $A$ be an approximately unital real Jordan operator algebra and let $\pi:A\to B(H)$ be a  contractive Jordan homomorphism. We let $P$ be the projection onto $K=[\pi(A)H]$. Then $\pi(e_t)\to P$ in the weak* (and WOT) topology of $B(H)$ for any J-cai $(e_t)$ for $A$. Moreover, for $a\in A$, we have $\pi(a)=P\pi(a)P,$ and the compression of $\pi$ to $K$ is a contractive Jordan homomorphism. Also, if $(e_t)$ is a partial cai for $A$, then $\pi(e_t)\pi(a)\to \pi(a)$ and $\pi(a)\pi(e_t)\to \pi(a)$. In particular, $\pi(e_t)|_K\to I_K$ SOT in $B(K)$.
\end{proposition}

\begin{proof}   As in the proof of Lemma 2.19 in \cite{BWj}.   \end{proof}

We will see in the proof of the next theorem that if $(a_t + i b_t)$ is a cai for $A_c$ in $\frac{1}{2} {\mathfrak F}_{A_c}$ (resp.\ ${\mathfrak r}_{A_c}$) then $(a_t)$
is a cai for $A$ in $\frac{1}{2} {\mathfrak F}_{A}$ (resp.\ ${\mathfrak r}_{A}$).

\begin{theorem}[Real case of Theorem 2.8 of \cite{BWj}] \label{frden} If $A$ is an approximately unital real Jordan operator algebra then $\mathfrak{F}_A$ is weak$^*$ dense in $\mathfrak{F}_{A^{**}}$ and  $\mathfrak{r}_A$ is weak$^*$ dense in $\mathfrak{r}_{A^{**}}$. Finally, $A$ has a partial cai in $\frac{1}{2}\mathfrak{F}_A$.
\end{theorem}
\begin{proof}  
Let ($A_c,\|\cdot\|_c$) be the operator space complexification of $A$. 
Then $\mathfrak{F}_{A_c}$ is weak$^*$ dense in $\mathfrak{F}_{A_c^{**}}$. Let $x\in \mathfrak{F}_{A^{**}} \subset
\mathfrak{F}_{A_c^{**}}$. By the density in the complex case, there is a net $(a_t+ib_t)$ weak$^*$ converging to $x$, which implies 
that $a_t$ weak$^*$ converges to $x$. Since $\|a_t-1\|\leq \|a_t+ib_t-1\|_c\leq 1$, we have $a_t\in \mathfrak{F}_A$. This shows that $\mathfrak{F}_A$ is weak$^*$ dense in $\mathfrak{F}_{A^{**}}$.

Similarly, if $x\in \mathfrak{r}_{A^{**}} \subset \mathfrak{r}_{A_c^{**}}$ then there is a net $(a_t+ib_t)$ in $\mathfrak{r}_{A_c}$ weak* converging to $x$. Since $(a_t+ib_t)+(a_t+ib_t)^*\geq 0$, we have $a_t+a_t^*\geq 0$. Moreover, $a_t$ weak* converges to $x$.

Finally, by the corresponding fact in the complex case, $A_c$ has a partial cai $(e_t+i \, f_t)$ in $\frac{1}{2}\mathfrak{r}_{A_c}$. Thus, $(e_t)$ is a partial cai in $A$. Since $\|1-\frac{1}{2}e_t\|\leq \|1-\frac{1}{2}(e_t+i \, f_t)\|\leq 1$, we have that $e_t\in \frac{1}{2}\mathfrak{r}_{A}$.
\end{proof}

A similar proof gives the analogue of Proposition 2.10 and Corollary 2.11 of \cite{BWj}:

\begin{proposition} \label{corde} Let $A$ be an approximately unital real Jordan operator algebra. Then the set of contractions in $\mathfrak{r}_A$ is weak* dense in the set of contractions in $\mathfrak{r}_{A^{**}}$.
\end{proposition}

\begin{proposition} If $A$ is a Jordan operator algebra with a countable Jordan cai, then $A$ has a countable partial cai in $\frac{1}{2}\mathfrak{F}_{A}$.
\end{proposition}

In the following result, which generalizes  \cite[Theorem 4.1]{BWj} (which in turn derives from \cite[Theorem 2.1]{BRord}),
we write  $x \preccurlyeq y$  to denote Re$(x) \leq {\rm Re}(y)$.   Here Re$(x) = (x+x^*)/2$.  Also ${\mathfrak c}_A = \Rdb^+ \, {\mathfrak F}_A$.

\begin{theorem} \label{brord} Let $A$ be a real Jordan operator algebra which generates a real $C^*$-algebra $B$,  and let 
${\mathcal U}_A$ denote the open unit ball $\{ a \in A : \Vert a \Vert < 1 \}$.  The following are equivalent:
\begin{itemize} \item [(1)]   $A$ is approximately unital.
 \item [(2)]  For any positive $b \in {\mathcal U}_B$ there exists  $a \in {\mathfrak r}_A$
with $b \preccurlyeq a$.
 \item [(2')]  Same as {\rm (2)}, but $a \in \frac{1}{2}  {\mathfrak F}_A$.
\item [(3)]   For any pair 
$x, y \in {\mathcal U}_A$ there exist   $a \in \frac{1}{2}  {\mathfrak F}_A$
with $x \preccurlyeq a$ and $y \preccurlyeq a$.
\item [(4)]    For any $b \in {\mathcal U}_A$  there exist  
$a \in \frac{1}{2}  {\mathfrak F}_A$
with $-a \preccurlyeq b \preccurlyeq a$. 
\item [(5)] For any $b \in {\mathcal U}_A$  there exist 
$x, y \in  \frac{1}{2}  {\mathfrak F}_A$ 
with $b = x-y$.   
\item [(6)]  ${\mathfrak r}_A$ is a generating cone (that is, $A = {\mathfrak r}_A - {\mathfrak r}_A$). 
\item [(7)]  $A = {\mathfrak c}_A - {\mathfrak c}_A$. 
\end{itemize}  
\end{theorem}  

\begin{proof}  (1) $\Rightarrow$ (2') \ 
Let $A$ be an approximately unital real Jordan operator algebra, and $b \in A_+$ with
$\| b \| < 1$.
Then $A_c$ is approximately unital, and by \cite[Theorem 4.1 (2)]{BWj} there exists $x+iy \in \frac{1}{2} {\mathfrak F}_{A_c}$ 
such that $b \leq {\rm Re} (x+iy)$.    This is easily seen to imply that $b \leq {\rm Re} (x)$.
Since $\| 1 - 2x \| \leq \| 1 - 2x - 2 iy \| \leq 1$ we have  $x \in \frac{1}{2} {\mathfrak F}_{A}$.
  
The other implications are as in \cite[Theorem 2.1]{BRord} and  \cite[Theorem 4.1]{BWj}, however some of the results invoked in those proofs
need to be replaced by their real variants from the present paper.  Also we note that the implications (2') $\Rightarrow$ (3) and (2) $\Rightarrow$ (6) follow 
from a fact from $C^*$-algebra theory.   Namely, from the Claim: if $x$ and $y$ (with $y = -x$ in the second implication) are selfadjoint and in ${\mathcal U}_B$ 
 then there is a positive element in ${\mathcal U}_B$ which is
 greater than both.   This is true in the real case too.  To see this we may assume that  $x$ and $y$ are also in $B_+$,
 by replacing them by $x_+$ and $y_+$.  In this case the Claim is usually an ingredient in standard proofs that $B$ has an increasing cai.
 However this case of the Claim also follows from the same fact but in the complex case.  Indeed there exists $a + ib \in B_c$ with (i)\
 $\| a + ib \| < 1$, and (ii)\ $x$ and $y$ dominated by $a + ib$.   However (i) implies that $\| a \| < 1$ and (ii) implies that
 $x$ and $y$ are dominated by $a$.   
\end{proof} 

A Jordan ideal  in a Jordan operator algebra $A$ is a subspace $J$ of $A$ with $J \circ A \subset J$.  

\begin{proposition} \label{Mid} Let $A$ be an approximately unital real operator algebra (resp.\ Jordan operator algebra).
\begin{itemize} 
\item [(1)]  If $A$ is weak* closed then the weak* closed ideals (resp.\ Jordan ideals) in $A$ which possess an identity
are in a bijective correspondence with the central projections $e \in A$, via $e \mapsto Ae$ (resp.\ $e \mapsto e \circ A$).
These ideals are $M$-summands of $A$.
\item [(2)]    The closed ideals (resp.\ Jordan ideals)  in $A$ which possess a cai  (resp.\ J-cai) 
are the subspaces of $A$ whose weak* closure in $A^{**}$ are of the form in {\rm (1)} for a central projection $e \in A^{**}$.   Thus they 
are of form $\{ x \in A : x = exe \}$ for some such $e \in A^{**}$. 
\end{itemize}  \end{proposition}
 
\begin{proof}  This follows just as in the complex case in e.g.\ \cite[Theorem 3.25]{BWj}.       \end{proof} 

We will do a more thorough study elsewhere of the $M$-ideals in real Jordan operator algebras (following on from   \cite[Section 5]{Sharma}).  

\begin{corollary} \label{quoi}   If $J$ is an approximately unital closed two-sided ideal (resp.\ Jordan ideal) in  a
 real operator algebra (resp.\ Jordan operator algebra) $A,$ 
then $A/J$ is (completely isometrically isomorphic to)
a real operator algebra (resp.\ Jordan operator algebra).   \end{corollary} 
 
 \begin{proof}      This follows from Proposition \ref{Mid} just as in the complex case in e.g.\ \cite[Theorem 3.27]{BWj}.  \end{proof} 
 
 Turning to one-sided ideals, Sharma showed in  \cite[Section 5]{Sharma} that the closed right ideals   in  an approximately unital 
 real operator algebra $A$ which possess a left cai  
are the subspaces of $A$ whose weak* closure in $A^{**}$ are of the form $eA^{**}$ for a projection $e \in A^{**}$.   Thus they 
are of form $\{ x \in A : x = ex \}$ for some such $e \in A^{**}$.    The projections $e$ occurring here are called {\em open projections}. 
The corresponding subspace $\{ x \in A : x = exe \}$ is called a {\em hereditary subalgebra} of $A$. 
For a real Jordan operator algebra $A$ we define a hereditary subalgebra to be a subspace
 of $A$ whose weak* closure in $A^{**}$ is of the form $eA^{**}e$ for a projection $e \in A^{**}$.   Thus they 
are of form $\{ x \in A : x = exe \}$ for such $e \in A^{**}$.  
We will discuss elsewhere  the noncommutative topology and hereditary subalgebras of real Jordan operator algebras, in the spirit of  
\cite{BNj}.

\begin{lemma} \label{juni}  If A is a nonunital approximately unital real Jordan operator algebra then the unitization $A^1$  is well defined up to completely isometric Jordan isomorphism, and the matrix norms are
$$\|[a_{ij}+\lambda_{ij} 1]\|=\sup\{\|[a_{ij}\circ c +\lambda_{ij} \, c]\|_{M_n(A)} : c \in {\rm Ball}(A)\}, \quad a_{ij}\in A, \lambda_{ij}\in \Rdb.$$ 
\end{lemma} 

\begin{proof} The proof is the same as for the complex case in Proposition 2.12 in \cite{BWj}.  \end{proof}

%{\bf Remark.}  
\begin{remark}   A unitization of a real Jordan operator algebra $A$ need  not be well defined/unique up to completely
isometric Jordan homomorphism.   This may be seen by modifying the argument as in Proposition 2.1 in \cite{BWj2} (see
\cite[Remark 4.4.3]{WTT}).   This means that some results about  real Jordan operator algebras with no kind of approximate
identity may not be treatable 
in the operator space category (as opposed to the Banach space category).   If however $A$ is approximately unital then  the last result 
shows that this problem does not exist.
\end{remark}

As in the complex case \cite[Theorem 2.8]{BWj}, a   real approximately unital Jordan operator algebra $A$ is an $M$-ideal in $A^1$. 

\begin{lemma}[Real case of Lemma 2.20 in \cite{BWj}] \label{blmf} Let $A$ be a real approximately unital Jordan operator algebra with a partial cai $(e_t)$. Denote the identity of $A^1$ by $1$. The following facts hold.
\begin{enumerate}	\item If $\psi:A^1\to \Rdb$ is a functional on $A^1$, then $\lim_t\psi(e_t)=\psi(1)$ if and only if $\|\psi\|=\|\psi|_{A}\|$.
	\item Let $\varphi \in A^*$. Then $\varphi$ uniquely extends to a functional on $A^1$ of the same norm.
\end{enumerate}
\end{lemma}
\begin{proof} (1) \ We have $(A^1)_c=(A_c)^1$ by Corollary \ref{uniqueJoaunitization}.
If $\psi:A^1\to \Rdb$ is a functional on $A^1$, then  $\|\psi\|=\|\psi_c\|$ by Proposition 1.4.1 in \cite{Li}.   By Lemma \ref{JOA_aprox_unital},
$(e_t)$ is a partial cai for $A_c$.
By Lemma 2.20 in \cite{BWj}, $\lim_t\psi(e_t)= \psi(1)$ if and only if $\|\psi_c\| =\|(\psi_c)_{|A_c} \|$.    Now $(\psi_c)_{|A_c} = (\psi_{|A})_c$,
and so $$\|(\psi_c)_{|A_c} \| = \| (\psi_{|A})_c \| = \| \psi_{|A} \|$$  by Proposition 1.4.1 in \cite{Li}.
Thus $\lim_t\psi(e_t)=\psi(1)$ if and only if $\|\psi\|=\|\psi|_{A} \|$.  

(2) \ This uses a similar idea: if $\psi, \rho$ are two Hahn-Banach extensions of $\varphi$  to  $A^1$,
then $\psi_c, \rho_c$ are two Hahn-Banach extensions of $\varphi_c$  to  $A_c^1$, by Proposition 1.4.1 in \cite{Li}.  
\end{proof}

\begin{lemma}[Real case of Lemma 2.21 in \cite{BWj}] For a norm $1$ functional $\varphi$ on an approximately unital real Jordan operator algebra $A$, the following are equivalent:
\begin{enumerate}
	\item $\varphi$ extends to a state on $A^1$.
	\item $\varphi(e_t)\to 1$ for every partial cai $e_t\in A$.
	\item $\varphi(e_t)\to 1$ for some partial cai for $A$.
	\item $\varphi(e)=1$ where $e$ is the identity of $A^{**}$.
	\item $\varphi(e_t)\to 1$ for every Jordan cai for $A$.
	\item $\varphi(e_t)\to 1$ for some Jordan cai for $A$.
\end{enumerate}
\end{lemma}

\begin{proof} The proof is the same as for the complex case in Lemma 2.21 in \cite{BWj}.  \end{proof}

The functionals on $A$ characterized in the last lemma are the {\em states} of $A$.

\section{Real positive elements and real positive maps} 

As we said in the introduction, if $A$ is a unital subspace or unital 
(Jordan) subalgebra of $B(H)$ then  then the set ${\mathfrak r}_{A}$ of real positive elements in $A$
does not depend on the 
particular $B(H)$ that $A$ sits in (isometrically and unitally).   A similar statement holds
 if $A$ is any (Jordan) subalgebra of $B(H)$.
Thus if $A$ is a Jordan operator algebra and $\pi : A \to B(K)$ is an isometric Jordan homomorphism
then, for example, $\pi(x) + \pi(x)^* \geq 0$ if and only if e.g.\ $\| 1 - tx \| \leq 1 + t^2 \| x \|^2$ for all $t > 0$.   Here $1$ is the identity of $A^1$, or the 
identity operator on $H$.  
This is a simple consequence of Meyer's theorem 
for Jordan operator algebras above.  Hence  $A \cap {\mathfrak r}_{A_c} = {\mathfrak r}_{A}$
if $A_c$ is any Jordan operator algebra complexification  of $A$.  Similarly  it is clear from Meyer's theorem above 
for Jordan operator algebras  that ${\mathfrak F}_A = A \cap {\mathfrak F}_{A_c}$.

In the following proof we write ${\rm Re}(x)$ for $\frac{1}{2} (x+ x^*)$.

\begin{lemma} \label{sprf}  If $X$ is a real 
unital operator space or Jordan operator algebra then 
${\mathfrak r}_X = \overline{\Rdb^+  {\mathfrak F}_X}$.
\end{lemma} 

\begin{proof}  If $\| 1 - x \| \leq 1$ then $\| 1 - {\rm Re}(x) \| \leq 1$.   Therefore 
$-1 \leq 1 - {\rm Re}(x) \leq 1$, and 
so ${\rm Re}(x)  \geq 0$.   Thus $\overline{\Rdb^+  {\mathfrak F}_X}
\subset  {\mathfrak r}_X$.  The reverse inclusion in the unital operator space case can be 
proved as is done in the complex case early in \cite[Section 2]{BBS}.  
If $X$ is a Jordan operator algebra then
${\mathfrak r}_X \subset {\mathfrak r}_{X_c} 
= \overline{\Rdb^+  {\mathfrak F}_{X_c}}$.     Suppose that $x \in {\mathfrak r}_X$ and 
$c_t \, x_t \to x$ with $c_t \in \Rdb^+$ and $x_t = a_t + i b_t \in {\mathfrak F}_{X_c}$.
Then $a_t \in X \cap {\mathfrak F}_{X_c} = {\mathfrak F}_{X},$ and $c_t \, a_t \to x$.   So $x \in \overline{\Rdb^+  {\mathfrak F}_X}$.
 \end{proof} 

%{\bf Remark.}  
\begin{remark} 
As in \cite{BRord,BWj} we may consider the ${\mathfrak F}$-transform: By \cite[Lemma 2.5]{BRord}, 
if  $x \in {\mathfrak r}_A$ for a real Jordan operator algebra $A$ then $${\mathfrak F}(x) = x(x+1)^{-1}
\in A \cap \frac{1}{2} {\mathfrak F}_{A_c} = \frac{1}{2} {\mathfrak F}_{A}.$$
Indeed let $D$ be the real operator algebra generated by 1 and $x$.  Since $x + x^* \geq 0$, the numerical range of $x$ is in the right half plane.  Hence the spectrum in $D_c$  of $x$ is in the right half plane.     Hence -1 is not in that spectrum, so $1+x$ has an inverse in $D_c$, in fact in $D$ (since e.g.\ if $x (a+ib) = 1$ then $xa = 1$).    
Then  ${\mathfrak F}(x) = x(1+x)^{-1} \in AD \subset A$.   Also, $1 - 2 x(1+x)^{-1} = (1-x)(1+x)^{-1}$, which is essentially the Cayley transform, which has norm $\leq 1$.
In fact this map has range  $U_A \cap \frac{1}{2} {\mathfrak F}_{A}$, where $U_A = \{ a \in A : \| a \| < 1 \}$, as in  \cite[Lemma 2.5]{BRord}. 
  Indeed suppose that  $w\in U_A \cap \frac{1}{2} {\mathfrak F}_{A} \subset U_{A_c} \cap \frac{1}{2} {\mathfrak F}_{A_c}$.
  We may suppose that $A$ is the closed algebra generated by $w$, which is an operator algebra.  Then there exists 
$x \in {\mathfrak r}_{A_c}$ with ${\mathfrak F}(x) = w$.  However $x = w (1-w)^{-1} \in A$.
Using the ${\mathfrak F}$-transform one may give another proof of Lemma \ref{sprf}   in the spirit of
e.g.\ \cite[Theorem 3.3]{BRII}.
\end{remark}

\begin{lemma} \label{antis}  Let $A$ be an operator system or real $JC^*$-algebra.   Then $x \in A$ is antisymmetric if and only if
$x \in {\mathfrak r}_A \cap (-{\mathfrak r}_A)$.
\end{lemma}

This is useful because there are many nice metric characterizations of ${\mathfrak r}_A$, as we said earlier 
when we defined that set (for example the conditions in \cite[Lemma 2.4]{BSan}).

\begin{lemma} \label{syjo}  If $X$ is a real operator system then $X = X_{\rm sa} \oplus X_{\rm as}$.  Also, $X_{\rm sa} = X_+ - X_+$, and 
$X = {\mathfrak r}_X - {\mathfrak r}_X = \Rdb^+ ( {\mathfrak F}_X - {\mathfrak F}_X)$.   \end{lemma} 

\begin{proof}  The first identity was established above Lemma  \ref{Tsyjo}. 
Let $x \in X_{\rm sa}$, then $$x = \frac{1}{2} (\| x \| 1 + x – ( \| x \| 1 - x ))
\in X_+ - X_+.$$     Similarly, since $\| x \| 1 - (\| x \| 1 \pm x)$ has norm $\leq \| x \|$, we have that 
$\| x \| 1 \pm x \in  \| x \|  {\mathfrak F}_X \subset \Rdb^+  {\mathfrak F}_X$.     The rest is clear.
 \end{proof}

If $X$ is a real operator system and $T : X \to B(H)$ we say that $T$ is {\em systematically real positive} if 
$x,y \in X$ with $x + y^* \in X_+$ implies that $T(x) + T(y)^* \geq 0$.

\begin{theorem}  \label{27}  Let $X$ be a real selfadjoint operator space  and let $T : X \to B(H)$ be real linear.  The following are equivalent:
\begin{itemize} \item [(i)]  $T$ is systematically real positive.
 \item [(ii)]  $T$ is both positive and selfadjoint. 
 \item [(iii)]  $T$ is real positive and selfadjoint. 
 \end{itemize}  \end{theorem}

\begin{proof} Clearly a selfadjoint map is positive if and only if it is real positive, and if these holds then 
$T : X \to B(H)$ is systematically real positive.   Conversely suppose that  $T$ is systematically real positive.
Let $x \in X_+$.   Then $T(x) + T(0)^* \geq 0$, so $T$ is positive.   
By Lemma \ref{syjo}  we have  $T(X_{\rm sa}) = T(X_+) - T(X_+) \subset B(H)_{\rm sa}$.  
 If $x^* = -x$ then $x + x^* = 0$, so that $T(x) + T(x)^*$ is both positive 
and negative.  Hence $T(x)^* = - T(x)$.   That is, $T( X_{\rm as}) \subset Y_{\rm as}$.   
So $T$ is selfadjoint by Lemma \ref{Tsyjo}.  
\end{proof}

\begin{theorem}  \label{db8}
If $T: A\to B(H)$ is a real positive linear map on a unital operator space $A$ whose restriction to 
$\Delta(A) = A \cap A^*$ is selfadjoint (or systematically real positive).  
Then the canonical extension $\tilde{T}: A+A^*\to B(H): x+y^*\mapsto T(x)+T(y)^*$ is well defined,  selfadjoint, and  positive.
\end{theorem}

\begin{proof}
Let $T : A \to B(H)$ be real positive with $T$ restricted to $\Delta(A)$ being
selfadjoint. 
Define $\tilde{T}(a + b^*) = T(a) + 
T(b)^*$
for $a, b \in A$.  
To see that   $\tilde{T}$ is well defined, suppose $a + b^* = x+ y^*$,  for $a, b, x, y \in A$.  Then $a-x = (y-b)^* \in \Delta(A)$, and 
so $$T(a-x) = T((y-b)^*) = (T(y) - T(b))^* .$$  Thus, $T$ is well defined.
If $z = a + b^*$ is positive (usual sense), then $$z = z^* = b + a^* = \frac{1}{2} (a + 
b^* + b + a^*)
= \frac{1}{2} (a + b) + (\frac{1}{2} (a + b))^*,$$ and $\frac{1}{2} (a + b) \in {\mathfrak r}_A$.  Since $T$ is real positive 
we have
 $$\tilde{T}(z) = T(\frac{1}{2}(a + b)) + T(\frac{1}{2}(a + b))^* \geq 0.$$   So $\tilde{T}$ is positive.
\end{proof}

     The converse of the theorem is true: 
 if $\tilde{T}$ is well defined,  selfadjoint, and  positive, then $T$ is real positive
 and its restriction to 
$\Delta(A) = A \cap A^*$ is selfadjoint  and systematically real positive.

If $X$ is a unital real operator space then we say that a real linear map $T : X \to B(H)$ is {\em systematically real positive} if 
$T$ extends to a positive selfadjoint map on $X + X^*$.    This is equivalent, by the theorem, to 
$T$ being real positive with $T$ restricted to $\Delta(A)$ being
selfadjoint.    It is also equivalent to: $x,y \in X$ with $x + y^* \geq 0$ implies that $T(x) + T(y)^* \geq 0$.   
One way to see the last equivalence is to note that this condition implies that  $T$ restricted to $\Delta(A)$
is systematically real positive, hence selfadjoint by Theorem  \ref{27}.  Then apply the last theorem.

Note that a real positive linear map from an operator system $X$ into $B(H)_{\rm sa}$ is
systematically real positive.   Indeed in this case $T(X_{\rm sa}) \subset B(H)_{\rm sa}$,
and as in the proof of Theorem  \ref{27} above $T(X_{\rm as}) \subset (0) \subset B(H)_{\rm sa}$.
So $T$ is selfadjoint by Lemma \ref{Tsyjo}.

We discuss now the meaning of   $A+A^*$ for a real operator algebra or real Jordan operator algebra $A$.   If  $A$ is also equipped with a compatible 
operator space structure (that is, if $A$ is a Jordan subalgebra of $B(H)$ and has the inherited matrix norms), then this is relatively unproblematic.
This is essentially the case treated in  \cite{WTT}.  The point is that a completely isometric (Jordan) homomorphism $\theta : A \to B$ between
real (Jordan) operator algebras extends to a completely isometric (Jordan) homomorphism between their complexifications.
By the complex theory this extends to a completely isometric (Jordan) homomorphism between the unitizations of the complexifications,
and then to a completely isometric  UCP map $(A_c)^1 + ((A_c)^1)^* \to (B_c)^1 + ((B_c)^1)^*$.   This restricts to a
completely isometric selfadjoint complete order isomorphism $A+A^* \to B + B^*$.    

On the other hand, if we treat $A$ without using matrix norms, that is, use morphisms that are isometric (Jordan) homomorphisms,
then it seems that the meaning of   $A+A^*$ is more problematic.   That is, we do not know at present if 
an isometric (Jordan) homomorphism $\theta : A \to B$ extends to an  isometric selfadjoint  order isomorphism $A+A^* \to B + B^*$. 
However  we will know shortly from Lemma \ref{ApA}  that it extends to a selfadjoint  order isomorphism $\tilde{\theta} 
: A+A^* \to B + B^*$.     There is a canonical norm on $A+A^*$
for which the latter map is an isometry (namely the norm inherited from $C^*_{\rm max}(A)$, the universal $C^*$-algebra for contractive 
(Jordan) homomorphisms from $A$, see \cite{WTT}), but we do not know yet if this norm always agrees with the one mentioned above in the last paragraph 
using a suitable operator space structure on $A$.    (It does follow from the later result Lemma \ref{rsu}  that   $\tilde{\theta}$ is isometric 
on the selfadjoint part of $A+A^*$.)    In any case, if we concern ourselves only with 
the order structure on $A+A^*$ there are no problems.      

\begin{lemma} \label{ApA}  Let $\theta : A \to B$ be a contractive (Jordan) homomorphism
between  real (Jordan) operator algebras.  Then $\theta$ extends uniquely to a  selfadjoint  positive map $\tilde{\theta} 
: A+A^* \to B + B^*$.    Also, restriction of $\theta$ to $\Delta(A)$ is a (Jordan) $*$-homomorphism into $\Delta(B)$.
If $\theta$ is also an isometric isomorphism onto $B$ then 
$\tilde{\theta}$ is a selfadjoint  order isomorphism onto $B+B^*$.  
\end{lemma}

\begin{proof}      We know that $\Delta(A)$ is a real $JC^*$-algebra.    Suppose that $B$ is a Jordan subalgebra of $B(H)$.
By Theorem \ref{contractive-hom},
 the restriction of $\theta$ to $\Delta(A)$ is a Jordan $*$-homomorphism into $B(H)$.
 In particular it is
selfadjoint, and maps into $\Delta(B)$.   The proof of Theorem  \ref{db8} now gives the first assertion.   Alternatively, one may extend $\theta$ to a unital contractive (Jordan) homomorphism between unitizations of $A$ and $B$ by  
Theorem \ref{Meyer-Real-Unique}, and then appeal to the {\em statement} of 
Theorem  \ref{db8}.  
The second assertion follows from the first applied to 
$\theta$ and $\theta^{-1}$.
\end{proof}

For a real operator algebra or real Jordan operator algebra $A$ then we say that a real linear map $T : A \to B(H)$ is {\em systematically real positive} if 
$T$ is real positive with $T$ restricted to $\Delta(A) = A \cap A^*$ being
selfadjoint.    It is also equivalent to the canonical extension $\tilde{T}: A+A^* \to B(H): x+y^*\mapsto T(x)+T(y)^*$ being well defined,  selfadjoint, and  positive.  
Indeed  the proof of Theorem  \ref{db8} shows that if $T$ is real positive with $T$ restricted to $\Delta(A)$ 
selfadjoint then $\tilde{T}$ is well defined,  selfadjoint, and  positive. 
Conversely, if the last condition holds then clearly $T$ is real positive and $T$ restricted to $\Delta(A) = A \cap A^*$ is 
selfadjoint.  
The latter implies that $x + y^* \geq 0$ implies that $T(x) + T(y)^* \geq 0$, but we are not sure if this condition is equivalent.

%\bigskip

%{\bf Remark.} 
\begin{remark}  Another class of maps that one could consider are the maps $T : A \to B$ that extend to a real positive 
map on a (Jordan) operator algebra complexification.    Then of course there are various variants of this class, such as 
contractions that extend to a real positive contraction on such a complexification.    
\end{remark}

\begin{lemma} \label{ispos}  Let $T : A \to B$ be a systematically real positive 
map between  real Jordan operator algebras.  Then $T(\Delta(A)) \subset \Delta(B)$, and $T$ restricts to a positive
selfadjoint  linear map 
from $\Delta(A)$ to $\Delta(B)$.   Thus  $0 \leq T(1) \leq 1$ if $A$ is unital and $T$ is contractive.
\end{lemma}

\begin{proof}    As above, $\tilde{T}: A+A^* \to B(H)$ 
is positive and selfadjoint.   Hence so is its restriction to $\Delta(A)$.   
\end{proof}

Let $T : X \to B(H)$ be a unital linear contraction on a unital operator space.   Then $T$ is real positive in the sense that 
$T$ takes ${\mathfrak r}_X$ to real positive operators.   This follows from the fact that 
${\mathfrak r}_X = \overline{\Rdb^+  {\mathfrak F}_X}$ (see Lemma \ref{sprf}).  
 However we shall see that $T$ need not be  systematically real positive.   Indeed if $X$ is a real operator system  then 
$T$ need not be selfadjoint,  although $T$ is antisymmetric (that is, $T( X_{\rm as}) \subset Y_{\rm as}$).    Indeed if $x^* = -x$ then $x + x^* = 0$, so that $T(x) + T(x)^*$ is both positive and negative.  Hence $T(x)^* = - T(x)$.

\begin{example} \label{expoly} Let $X$ be the polynomials with real coefficients of degree $\leq 1$ in $C([0,1])$.  
Let $x(t) = t$, and let $z$ be the matching monomial in the disk algebra.    Let $g = \frac{1+z}{2}$, a 
contraction in the disk algebra, and define $T(s + tx) = s + t g$ for $s, t \in \Rdb$.   We claim that 
$T$ is a unital contraction on $X$ which is not selfadjoint, nor systematically real positive, nor positive.  
 Indeed the norm of $s + t g$ in the disk algebra
is easily seen to be $|s+t/2| + |t|/2$, whereas it is an exercise that the norm of $s + t x$ in $X$ is $\max \{ |s|, |s+t| \}$.
Finally, by considering the cases that $s, t,$ and $s+t/2$ are positive and negative, one can prove easily that 
$$|s+t/2| + |t|/2 \, \leq \, \max \{ |s|, |s+t| \} \, , \qquad s, t \in \Rdb.$$
That is, $T$ is a unital  contraction on $X$.  It clearly is not selfadjoint, hence is not systematically real positive.

This example illustrates some other points.   First, although $T$ is a unital  contraction on $X$, and although
$X$ is so very simple, $T$ is not a complete contraction.  Indeed if it were then by \cite[Theorem 2.1]{RComp} 
it  extends to a unital  complete contraction on the complexification.    This extension would have to be selfadjoint
by the well known complex case \cite{Arv,Pau}, 
giving the contradiction that $T$ is selfadjoint.   Also, there is not a real version of the 
completely contractive version of von Neumann's inequality (sometimes attributed to Sz-Nagy,
and following from the Sz-Nagy dilation), else $T$ would be completely contractive.   There is
a real form of von Neumann's inequality.  (Indeed if  $T$ is a contraction in $B(H) \subset B(H)_c$ then the map 
 $p + \bar{q} \mapsto p(T) + q(T)^*$ is a positive  (hence completely positive, by \cite[Theorem 3.11]{Pau}) unital 
 contraction from a dense selfadjoint subspace of $C(\Tdb,\Cdb)$ into $B(H)_c$.
 The restriction to the  set of $p + \bar{q}$ for polynomials $p, q$ with real coefficients, is a 
 positive selfadjoint contraction into $B(H)$.)

Note that $R(s + tx) =  T(s + tx) \oplus (s + tx) \in A(\Ddb) \oplus^\infty X$ is a unital isometry
which is not systematically real positive, nor selfadjoint, nor positive.
Thus if  $X$ is a unital operator space, $X + X^*$  need not be `well defined' as an ordered Banach space.   Indeed in the 
above example $X$ and $T(X)$ are `isometrically the same' as unital operator spaces via  the unital isometry $T$.
However  $X + X^*$ has dimension 2, while $T(X) + T(X)^*$  has dimension 3, so $T$ 
certainly does not extend to a faithful map on $X + X^*$, let alone a positive selfadjoint isometric one, nor  an order isomorphism onto its range.  
  On the other hand $X + X^*$  is `well defined' as an ordered linear space if we use morphisms on a unital operator space $X$ that are 
real positive unital isometries which is selfadjoint on $\Delta(X)$.   Indeed if $T : X \to Y$ is a surjective unital isometry
between unital operator spaces $X$ and $Y$, and if $T$ and $T^{-1}$ are systematically real positive, then 
 the canonical extension $\tilde{T}: A+A^*\to B(H): x+y^*\mapsto T(x)+T(y)^*$ is  selfadjoint and an order embedding.  
 Also, $X + X^*$  is `well defined' as an operator system, if we use morphisms on a unital operator space $X$ that are 
 unital complete isometries.    See the remark above Theorem \ref{contractive-hom}. 
\end{example}

\begin{example}  \label{rpne} A  real positive map, or systematically real positive map, need not extend to a real positive map  on a complexification.
A positive selfadjoint map on a real operator system need not extend to a  positive map  on a complexification. 
Also,  a unital linear contraction need not extend to a real positive map  on a fixed complexification.  
For an example of these, we proceed as in Proposition \ref{wtex}.  
Let $u : Y \to X$ be a linear contraction  that does not extend to a 
contraction from $Y_c$ to $X_c$.     Let  $\theta_u$ be the canonical extension 
of $u$ to a unital contractive homomorphism  ${\mathcal U}(Y) \to {\mathcal U}(X)$.   
As we said in Proposition \ref{wtex}, by Theorem \ref{db8}, or by the real variant of the Paulsen lemma in \cite[Lemma 4.12]{Sharma} (see also p.\ 492
in \cite{ROnr}), 
$\theta_u$ extends to a positive selfadjoint  map $\theta_u + (\theta_u)^*$ on  the Paulsen system ${\mathcal S}(Y) = 
{\mathcal U}(Y) + {\mathcal U}(Y)^*$ (which is even
real contractive by Lemma \ref{rsu}).   Thus $\theta_u$ is systematically real positive.  
But it does not extend to a real positive map  on the complexification, by the argument in 
Proposition \ref{wtex}.   
 \end{example}

\begin{example}  A positive functional even on a real $C^*$-algebra need not be selfadjoint nor real positive.
The example above Proposition 4.1 in \cite{ROnr} shows this: apply that functional 
to the real positive matrix with 
all entries $1$ except for a -3 in the $1$-$2$ corner.
Note that if we scale this example  to be a unital functional $\psi$, then $\Vert \psi \Vert \geq 1$.
But in fact $\Vert \psi \Vert > 1$ since if $\Vert \psi \Vert = 1$ then $\psi$ would be selfadjoint
by the next result.  \end{example}  

\begin{example}  \label{f} A positive unital selfadjoint real linear 
map on a complex operator system need not be a contraction.  There is a $2 \times 2$ matrix counterexample due to Arveson 
see e.g.\ A.2 in \cite{Arv}.   Viewing this as a 
real operator system this map is unital and real positive, but not a contraction.  

Indeed a positive unital selfadjoint real linear 
map on a real $JC^*$-algebra need not be bounded.   For an example of this let $E$ be an infinite dimensional
space of selfadjoint operators on a Hilbert space $H$ with $x^2 \in \Rdb I_H$ for all $x \in E$.   Let  $A$
be the set of matrices 
in $M_2(B(H))$ with diagonal entries in $\Rdb I_H$, and off-diagonal entries $x$ and $-x$ for $x \in E$.
Note that $A$ is a real $JC^*$-algebra.  We have $A_{\rm sa} = (\Rdb I_H) \oplus (\Rdb I_H)$,
the positive elements in $A$ are $(\Rdb_+ \, I_H) \oplus (\Rdb_+ \, I_H)$, and 
$A_{\rm as}$ consists of the matrices in $A$ with zero main diagonal entries.  
If $T : E \to B(K)_{\rm sa}$ is an unbounded real linear map let $\theta_T$ be the map on $A$
taking $$\begin{bmatrix}
       \lambda \,  I     & x\\
        -x   & \mu \, I
    \end{bmatrix} \; \mapsto \; \begin{bmatrix}
        \lambda \, I     & Tx\\
        -Tx   & \mu \, I
    \end{bmatrix} .$$ This is a positive unital selfadjoint real linear 
map   on a real $JC^*$-algebra which is not  bounded.    If $K = \Rdb$ in   Example  \ref{f}  then $\theta_T : A \to M_2$.
This positive map is not $2$-positive, otherwise by Lemma \ref{inMn} it would be completely positive, and hence 
completely contractive by Lemma \ref{lemos}.     \end{example} 

\begin{lemma} \label{sfun}    For a functional $\varphi$ on a unital real operator space  or 
approximately unital real Jordan operator algebra $X$ the following are equivalent:
\begin{itemize} \item [(i)] $\varphi$ is real positive.
\item [(ii)]  $\varphi$ is systematically real positive.  
 \item [(iii)]  $\varphi$ is real completely  positive (RCP). 
 \item [(iv)] $\varphi$ is a nonnegative multiple of a state.  
 \end{itemize} 
Such functionals are bounded with $\| \varphi \|  = \| \varphi \|_{\rm cb}$.
This equals $\varphi(1)$, or $\lim_t \, \varphi(e_t)$ in the case of 
a cai $(e_t)$.  
 If $\varphi$  is unital then the above equivalent conditions
 hold iff $\varphi$ is contractive. 
 If $X$ is a real operator system then a unital functional $\varphi$ on $X$  is contractive if and only if it is  positive and selfadjoint; such a functional  is  completely positive.  
  \end{lemma} 

\begin{proof}    Suppose that $\varphi$ is a real positive functional on  a unital real operator space $X$.   Its restriction to 
$\Delta(X) = X \cap X^*$ is  a 
real positive functional on an operator system.   If $x^* = -x$ in $\Delta(X)$ 
then as in the proof of Theorem  \ref{27} we have that $\varphi(x) = 0$.  
So    $\varphi$ is  selfadjoint on $\Delta(X)$ by Lemma \ref{Tsyjo}.   
Hence  $\varphi$ is systematically real positive by Theorem  \ref{db8}.  
Such maps are  real completely positive, as in the complex case.    One way to see this
is if $\tilde{\varphi}$ is the canonical extension to $X + X^*$,
and if $x = [x_{ij}] \geq 0$ in $M_n(X + X^*)$ then $[\tilde{\varphi}(x_{ij})]$ is a selfadjoint matrix.
We have $$\langle [\tilde{\varphi}(x_{ij})] \xi , \xi \rangle 
= \tilde{\varphi}(\xi^T x \xi) \geq 0 .$$  So $[\tilde{\varphi}(x_{ij})] \geq 0$.
So $\tilde{\varphi}$ is completely positive and hence $\varphi$ is real completely positive.  
Since $\tilde{\varphi}$ is completely positive  it extends to a 
 completely positive, hence completely bounded, map on the complexification (see e.g.\ Proposition \ref{rcps}).  
 Indeed the norm and  completely bounded norm of this extension is 
 $\varphi(1)$ by e.g.\ \cite[Proposition 3.6]{Pau}, hence $\| \varphi \|  = \| \varphi \|_{\rm cb} = \varphi(1)$.

If $A$ is an  approximately unital real Jordan operator algebra then 
$A = {\mathfrak r}_A - {\mathfrak r}_A$ by Theorem \ref{brord}.   
If   $\varphi$ is a real positive functional then the argument of \cite[Corollary 2.8]{BRord} shows that $\varphi$ is bounded.
Hence $\varphi^{**}$ is real positive, and by the above $\varphi^{**}$ is systematically real positive and 
real completely  positive (RCP), 
with  $\| \varphi^{**} \|  = \| \varphi^{**} \|_{\rm cb} = \varphi^{**}(1)$. 
Hence $\| \varphi \|  = \| \varphi \|_{\rm cb} = \lim_t \, \varphi(e_t)$, the 
latter since $e_t \to 1$ weak* for a cai $(e_t)$.  

A contractive unital functional on  a unital real operator space $X$ extends to a
contractive unital functional on  a 
real $C^*$-algebra.   This is positive and selfadjoint by \cite[Proposition 5.2.6  (3)]{Li}, 
and so its restriction is systematically real positive, and indeed positive 
if $X$ is an operator system.   

In particular any state on  a unital real operator space is real positive.
Similarly on an approximately unital real Jordan operator algebra (e.g.\ by taking the bidual and using the same
argument).  Conversely, if $\varphi$ is real positive and nontrivial then $\frac{1}{\alpha} \, \varphi$ is a state, where
$\alpha$ is $\varphi(1)$, or $\lim_t \, \varphi(e_t)$ in the case of 
a cai $(e_t)$.  

A selfadjoint functional on an operator system is clearly real positive if and only if it is positive,
and is a positive multiple of a state.  So if it is unital it is a state and has norm $1$.  
Such maps are  real completely positive and completely positive,  as in the complex case--see e.g.\
the argument a few paragraphs above.    
\end{proof}

%{\bf Remark.}
\begin{remark}  If $A$ is a real $JC^*$-algebra  then a positive and selfadjoint
 functional  is real positive, systematically real positive, and completely positive, and 
is RCP.     \end{remark}

%\bigskip

A unital linear contraction on a unital operator space  is real positive 
as was stated above.   However the converse is false in general, see Example \ref{f}.   
Nonetheless, we have:

 \begin{lemma} \label{sfunfun}    For a  real linear map  $T : X \to B$ from a unital real operator space or approximately unital real Jordan operator algebra 
 into a 
  commutative real $C^*$-algebra the following are equivalent:
\begin{itemize} \item [(i)] $T$ is real positive.
\item [(ii)]  $T$ is systematically real positive.  
 \item [(iii)]  $T$ is real completely  positive. 
 \end{itemize} 
Such maps are bounded with $\| T \|  = \| T \|_{\rm cb}$.
This equals $\| T(1) \|$, or $\lim_t \, \| T(e_t) \|$ in the case that $A$ has 
a J-cai $(e_t)$.  
If $T$ is unital  then the above hold iff $T$ is contractive.   If $X$ is a real operator system and $T$ is unital 
 then $T$  is contractive if and only if it is  positive and selfadjoint;
such a map is completely positive.  
  \end{lemma} 

\begin{proof}    This follows from Lemma \ref{sfun}  and the usual trick for this 
in the complex case.  Indeed the commutative real $C^*$-algebra
may be replaced by $C(K,\Cdb)$ viewed as a real $*$-algebra by basic facts about  commutative real $C^*$-algebras
 (see \cite[Theorem 1.9]{Ros} or \cite{Li}).
Then apply  Lemma \ref{sfun} to the linear functional $\psi_w = T(\cdot)(w)$ for fixed  $w \in K$.   This will be real positive 
if $T$ is real positive, and so $\|  \psi_w \| = \psi_w(1)$ by Lemma \ref{sfun}.  
Thus $\| T \|$ equals
$$\sup \{ |T(f)(w) | : w \in K, f \in {\rm Ball}(X) \} = \sup \{ |T(1)(w) | : w \in K, f \in {\rm Ball}(X) \}$$
which is  $\| T(1) \|.$  In particular $T$ is bounded.  
 For an approximately unital real Jordan operator algebra $A$ we can take a weak* continuous extension on $A^{**}$ to 
 reduce to the unital case as in e.g.\  Proposition \ref{rcps}.  
 
  If $T$ is a unital  contraction then $T$ is real positive as we said above Example \ref{expoly}.   If further $X$ is a real operator system then 
  $T$ is positive and selfadjoint by (ii).   It is also completely contractive since $\| T \|  = \| T \|_{\rm cb}$, so applying the above to each 
  $T_n$ we see that it is completely positive.  
 \end{proof}

\begin{lemma} \label{inMn} 
Let $T : A \to M_n$ be a  linear map on a unital operator space, or a bounded linear map on an approximately unital real Jordan operator algebra,
which is real $n$-positive (that is, $T_n$ is real positive).
Then  $T$ is RCP and systematically real positive,
and $\| T \|_{\rm cb} = \| T \|$.    The latter equals $\| T (1) \|$ if $A$ is unital, otherwise equals
$\| T^{**}(1) \| = \lim_t \, \| T(e_t) \|$, if $(e_t)$ is a J-cai for $A$.    \end{lemma}

\begin{proof}    Suppose that $x = [x_{ij}]$ and $x + x^*  \geq 0$ in $M_m(A + A^*)$.
Then $[T(x_{ij})] + [T(x_{ji})^*] = T_m(x) + T_m(x)^*$ is certainly selfadjoint.   So to test if this is positive it is enough to check that
$\langle (T_m(x) + T_m(x)^*) \eta , \eta \rangle \geq 0$ for $\eta \in (\Rdb^n)^m$ and $m > n$.   
However in a real Hilbert space $\langle (z+z^*)   \eta , \eta \rangle = 2 \langle z  \eta , \eta \rangle$.
Hence  it is enough to check that $\langle T_m(x)  \eta , \eta \rangle \geq 0$.
As in  the proof of Proposition 2.2.2 in 
 \cite{ER} there is an isometry $\alpha : \Rdb^n \to \Rdb^m$ and $\xi \in (\Rdb^n)^n$ such that 
 $\eta = (\alpha \otimes I_n) (\xi)$.   
 Set  $y = \alpha^* x \alpha$ then $y \in M_n(A)$ and 
 $$2 \langle T_n(y)   \xi , \xi \rangle =  \langle (T_n(y) +   T_n(y)^*) \xi , \xi \rangle \geq 0$$
 since $T_n$ is real positive.   Hence as in Proposition 2.2.2 in \cite{ER}, 
 $$ \langle T_m(x) \eta , \eta \rangle = \langle T_n( \alpha^* [x_{ij}] \alpha) \xi , \xi \rangle =  \langle T_n(y)   \xi , \xi \rangle \geq 0.$$
Thus $T$ is RCP.    
The rest follows from  Proposition \ref{rcps}.   \end{proof} 

%{\bf Remark.} 
 \begin{remark} There is a {\em  Schwarz inequality} for  2-positive real linear maps on real $C^*$-algebras, proved identically to 
e.g.\ \cite[Proposition 3.3]{Pau}.  See p.\ 492 in \cite{ROnr}  for the Schwarz inequality for real UCP maps.
\end{remark}

%\bigskip 
 
 The map ${\rm Re} : B(H) \to B(H)_{\rm sa}$  is real positive, positive, unital, contractive 
and selfadjoint.      
  
We say that a map $u : X \to Y$  is  {\em real contractive}  if $\| {\rm Re} \, u(x) \| \leq \| {\rm Re} \, x \|$ for $x \in X$.
We say that $u$  is  {\em real bounded} if there is a constant $c \geq 0$ with $\| {\rm Re} \, u(x) \| \leq  c \, \| {\rm Re} \, x \|$ for $x \in X$, 
and then we write $\| u \|_r$ for the  least $c$ in this inequality.  This is called the 
{\em real bounded norm} (actually it is a seminorm).

\begin{lemma} \label{rsu}     Let  $A$ be a real unital operator space.
  If $u : A \to B(H)$ is real positive and restricts to  selfadjoint map on $\Delta(A)$ then 
 $u$  is 
 real bounded with $\| u \|_r  = \| u(1) \|$.  
 Indeed  $u$ extends   to a positive selfadjoint  map 
$\tilde{u} : {\mathcal S} = A + A^* \to B(H)$,  with  $\|  \tilde{u}  \|_r = \| u(1) \|$.   

If $A$ is an approximately unital real Jordan operator algebra
and $u : A \to B(H)$ is real positive and restricts to  selfadjoint map on $\Delta(A)$ then 
 $u$  is 
 real bounded with $\| u \|_r  = \sup_t \,  \| u(e_t) \|$, where 
 $(e_t)$ is any Jordan cai for $A$.   \end{lemma}

\begin{proof}     First assume $A$ is unital. 
The restriction of $u$ to $\Delta(A)$ is real positive, so by Theorem \ref{db8} 
we have that $u$ is systematically real positive, $u(1) \geq 0$,
and $u$ extends to a positive selfadjoint 
$\tilde{u} : {\mathcal S} = A + A^* \to B(H)$.
For any unit vector $\xi \in H$, by Lemma \ref{sfun}  
we have that $\varphi_\xi = \langle u (x) \xi , \xi \rangle$ is systematically real positive and bounded with norm
$\langle u(1)  \xi , \xi \rangle$.  By the proof of Lemma \ref{sfun}, $\varphi_\xi$ extends to 
a positive selfadjoint functional $\psi_\xi$ on ${\mathcal S}$ of norm $\langle u(1)  \xi , \xi \rangle$.
 
For $x \in {\rm Ball}({\mathcal S})$ we have $\| {\rm Re} \, x \| \leq \| x \| \leq 1$. 
We have ${\rm Re}  \, \tilde{u} (x) =  \tilde{u} ({\rm Re} \, x)$ and 
$\| {\rm Re}  \, \tilde{u} (x) \|$ equals $$\sup \,  | \langle \tilde{u} ({\rm Re} \, x) \xi , \xi \rangle |  
= \sup \, | \psi_\xi ({\rm Re} \, x) |$$
(suprema over $\xi \in H , \| \xi \| = 1$).  This is dominated by $\sup \, \langle u(1)  \xi , \xi \rangle = \| u(1) \|$.

Next, if $A$ is an approximately unital real Jordan operator algebra then the same argument gives that 
$\varphi_\xi$ is systematically real positive and bounded with norm
$\sup_t \, |\langle u(e_t)  \xi , \xi \rangle|$.     Then 
$$\| {\rm Re}  \, \tilde{u} (x) \| \leq \sup \,  | \langle u(e_t) \xi , \xi \rangle | \leq
\sup_t \,  \| u(e_t) \| ,$$
where the second last supremum is over $t$ and $\xi \in H , \| \xi \| = 1$.
 \end{proof}

\begin{corollary} Let  $A$ be a unital real $JC^*$-algebra.
If $u : A \to B(H)$ is selfadjoint and positive then  
$u$  is real bounded with $\| u \|_r = \| ({\rm Re} \; u)^{**}(1) \| = \lim_t \, \| u(e_t) \|$.  
Here $(e_t)$ is a (increasing, if one wishes) cai for $A$.  \end{corollary} 

\begin{proof}  Since $u$ is real positive  we may appeal to Lemma \ref{rsu}.     \end{proof}

\begin{corollary} \label{trivj} Let $A, B$ be  approximately unital real Jordan operator algebras,
and let  $T : A \to B$ be a contraction which is approximately unital (that is, takes some Jordan cai to a Jordan cai), or more generally for which $T^{**}$ is unital.    Then $T$ is real positive.    If in addition  
$T$ is selfadjoint on $\Delta(A)$ then $T$ is systematically real positive.

 If $\theta : A \to B$ is a contractive Jordan homomorphism then $\theta$ is systematically real positive.
  \end{corollary}

\begin{proof}    By taking the second dual we may assume that $A, B$ are unital, and that $T(1) = 1$.    
Then the first assertion follows from the lines before Example \ref{expoly}.    The `in addition' statement then 
follows from Theorem  \ref{db8} or the paragraphs after that.
The last assertion follows easily from the first after replacing $B$ with $\overline{\theta(A)}$, and noting that the restriction of $\theta$ to the real JC*-algebra $\Delta(A)$ is 
selfadjoint by Theorem \ref{contractive-hom}.    It also follows easily from  Lemma \ref{ApA}.
\end{proof} 

%{\bf Remark.} 
\begin{remark}
 If $\theta : A \to B$ is a contractive (resp.\ isometric) Jordan homomorphism between unital or approximately unital real Jordan operator algebras
then we are not certain if $\tilde{\theta} : A + A^* \to B + B^*$ is contractive
(resp.\ isometric).  \end{remark}

\begin{theorem} \label{cepro}   Let $A$ and $B$ be approximately unital  real Jordan operator algebras, and write
$A^1$ for a real Jordan operator algebra unitization of $A$ with $A \neq A^1$.  Let $C$  be a unital real Jordan operator
algebra containing $B$ as a closed Jordan subalgebra.
\begin{itemize}  \item [(1)]
A  real positive real contractive  linear map  $T : A \to B$ extends to a unital real
 positive  linear map from $A^1$ to $C$, which is systematically real
 positive and real contractive if $T$ is also selfadjoint on 
$\Delta(A)$.   
 \item [(2)]  A  real completely positive completely contractive  linear map  $T : A \to B$ extends to a unital real
 completely  positive  completely contractive linear map from $A^1$ to $C$. 
\end{itemize}
\end{theorem}  

\begin{proof}   (1) \ We follow the proof of \cite[Theorem 2.3]{BNjp}, with a few tweaks.
Clearly (2) follows from (1).   Let 
$\tilde{T} : A^1 \to C$ be the canonical unital extension of $T$ in (1), and write $e, f$ for the units of $A^1$ and $C$.
So $\tilde{T}(a + s e) = T(a) + s f$ for $s \in \Rdb, a \in A$.    
Suppose that $A$ is a Jordan subalgebra of some real $C^*$-algebra $D$.   Since $e \notin A$ we may assume that $e = 1_{D^1} \notin D$. 
Suppose that Re $(x + \lambda e) \geq 0$ for  $x \in A$ and  scalar $\lambda$.    We need to prove that 
Re $(T(x)  + \lambda f) \geq 0$.   This is clear if Re $(\lambda)  = 0$, so suppose the contrary.  
Now Re $(\lambda) > 0$ (by considering the character 
$\chi$ on $D^1$ that annihilates  $D$; this
is a state so that Re$(\chi(x) + \lambda) = {\rm Re}(\lambda) \geq 0$).   
 Since Re$(x + \lambda e) \geq 0$ we have $-\frac{1}{{\rm Re} (\lambda)} \, {\rm Re} (x) \leq e$.   
Let $$x_n = -   \frac{n-1}{n \, {\rm Re} (\lambda)} \, x \, , \; \; \; 
y = {\rm Re} (x_n) \leq   \frac{n-1}{n} \, e$$ and $z =y _+ 
\leq  \frac{n-1}{n} \, e$.   By Theorem \ref{brord} there exists 
a contraction $a \in A$ with $0 \leq z \leq {\rm Re} (a) \leq e$.   Now 
 Re $(a - x_n)   \geq 0$, since  Re $(x_n) = y \leq y_+ = z  \leq {\rm Re} (a)$.
Also 
 $\| {\rm Re} \, (T(a)) \| \leq 1$  since $a$ and ${\rm Re} \, T$ are contractions, and therefore  $0 \leq {\rm Re} \, (T(a)) \leq f$.  Also, Re $T(a - x_n) \geq 0$, 
so that  Re $(T(x_n)) \leq {\rm Re} \, (T(a)) \leq f$.   That is, 
$$ -  \frac{n-1}{n \, {\rm Re} (\lambda)} \, {\rm Re} \, (T(x))  \leq f.$$
Letting $n \to \infty$ we have that ${\rm Re} \, (T(x) +  \lambda f) \geq 0$ as desired.  
Hence $\tilde{T}$ is  a unital real positive  map, which is selfadjoint on $\Delta(A^1)$,   
 and thus is  real contractive by  Lemma \ref{rsu}. 
 
 (2) \ We have that $T_c : A_c \to B_c$ is completely positive and completely contractive.  Then apply 
  \cite[Proposition 2.2]{BNp}, and finally restrict to $A^1$ (since $(A^1)_c = (A_c)^1$). 
 \end{proof}
 
%{\bf Remarks.} 
\begin{remark} Of course the extensions in the previous result are unique.   As in the complex case \cite{BNp,BNjp} one may apply these results to extend 
projections $P : A \to A$ to  unital projections on $A^1$.  

In the complex scalar case one gets a better result   \cite[Theorem 2.3]{BNjp}: A  real positive contractive  linear map  $T : A \to B$ extends to a unital real
 positive  contractive  linear map on  $A^1$.     We do not know if this is true in the real case, even if $T$ is systematically real
 positive.  If it were one would obtain the corollary that 
 a bounded real linear map $T : A \to B$ between approximately unital Jordan operator algebras which is selfadjoint on $\Delta(A)$, 
is  real positive and contractive if and only if $T({\mathfrak F}_A) \subset {\mathfrak F}_B$.   We do not know if this is true either, although certainly
$T({\mathfrak F}_A) \subset {\mathfrak F}_B$ implies that $T$ is real positive.  
\end{remark}

%\bigskip 

%{\em Acknowledgements.}   
\subsection*{Acknowledgment} 
Several  results here (in particularly, many in Sections 2--4) are from the May 2020 PhD thesis of  author W.\ Tepsan \cite{WTT}.  
 Other complementary facts, alternative proofs, and additional theory may be found there,
 and there may also possibly be a   forthcoming paper on that topic.
  
  We thank M. Kalantar for several helpful discussions, which we mention in more detail
 in and around  Theorem
\ref{ijco} (see also \ref{coisre}).   We also thank \'Angel Rodr\'iguez Palacios for some
comments.

  In terms of future 
 directions, the noncommutative topology of real Jordan operator algebras in the spirit of 
\cite{BNj} 
looks like a fruitful topic that should be pursued elsewhere, as well as some other 
features of the real positive cone that have not been explored for the real case here or in 
\cite{WTT}.    There are no doubt some interesting `operator space
aspects' of the theory of real associative operator algebras that are worth developing.  
In addition, there are several open questions stated in this paper.   Finally it would be very 
worthwhile to find many more interesting examples of real Jordan operator algebras, particularly if they might be important in quantum physics.   It does not seem hard to find examples of real Jordan operator algebras beyond those already mentioned in this paper or \cite{WTT}.  For example if one looks inside the  upper triangular (real or complex) matrices ${\mathcal U}_n$ one quickly spots 
several  new examples.  E.g.\ the matrices in ${\mathcal U}_2(\Cdb)$ with $a_{22} = \overline{a_{11}}$  and $a_{12}$ real.  Or certain matrices in ${\mathcal U}_n$ that are `symmetric' with respect to the diagonal connecting $a_{n1}$ to $a_{1n}$.


\begin{thebibliography}{99} 
\bibitem{ARU} Sh. A. Ayupov, A. A. Rakhimov, and Sh. M. Usmanov, {\em Jordan, Real, and Lie Structures in Operator Algebras,} vol. 418, MAIA, Kluwer Academic Publ., 1997. 

\bibitem{ARU2} Sh. A. Ayupov, A. A. Rakhimov, and   A. Kh.\  Abduvaitov, 
{\em Real $W^*$-algebras with an abelian Hermitian part (Russian),}
Mat.\ Zametki {\bf 71} (2002),  473--476; translation in 
Math.\ Notes  {\bf 71} (2002), 432--435. 

\bibitem{Arv}   W. B. Arveson, {\em Subalgebras of $C^{*}$-algebras,}  Acta Math.\  {\bf 123} (1969), 141--224.  

\bibitem{BBS}  C. A. Bearden, D. P. Blecher and S. Sharma, {\em On positivity and roots in operator algebras,}  Integral Equations Operator Theory {\bf  79} (2014),  555--566.

\bibitem{BSan} D. P. Blecher,  {\em Generalization of C*-algebra methods via real positivity for operator  and Banach algebras,}  
 pages 35--66
in "Operator algebras and their applications: A tribute to Richard V. Kadison", (ed.\ by R.S.\ Doran and E.\ Park),
vol. 671, Contemporary Mathematics, American Mathematical Society, Providence, R.I.\  
2016.

\bibitem{Bposx} D. P. Blecher,  {\em Real positive maps and conditional expectations on operator algebras,}  p.\ 
63--102 in {\em Positivity and its applications}, (Eds. E. Kikianty, M. Mabula, M. Messerschmidt, J. H. van der Walt, and M. Wortel),  Trends in Mathematics, Birkh\"auser (2021).

\bibitem{BLM}  D. P. Blecher and C.  Le Merdy, {\em Operator algebras and their modules---an
operator space approach,} Oxford Univ.\  Press, Oxford (2004).

 
\bibitem{BNp}   D. P. Blecher
and M. Neal, {\em Completely contractive projections on operator algebras,}  Pacific J. Math.   {\bf 283} (2016), 289--324.


\bibitem{BNj}   D. P. Blecher
and M. Neal, {\em  Noncommutative topology and Jordan operator algebras,}  Mathematische Nachrichten {\bf 
292} (2019), 481--510. 

\bibitem{BNjp}   D. P. Blecher
and M. Neal, {\em Contractive projections and real positive maps on  operator algebras,}   Studia Math {\bf 256} (2021), 21--60. 

\bibitem{BP} A. B\"ottcher and A.  Pietsch, {\em  Orthogonal and skew-symmetric operators in real Hilbert space,}
 Integral Equations Operator Theory {\bf 74} (2012), 497--511. 
 
\bibitem{BRI}  D. P. Blecher and C. J. Read, {\em  Operator algebras with contractive approximate identities,}
J. Functional Analysis {\bf 261} (2011), 188--217.

\bibitem{BRII}  D. P. Blecher and C. J. Read, {\em  Operator algebras with contractive approximate identities II,} J. Functional Analysis {\bf 264} (2013), 1049--1067.

\bibitem{BRord}  D. P. Blecher and C. J. Read, {\em   Order theory  and interpolation in operator algebras,}   Studia Math. {\bf 225} (2014), 61--95.

\bibitem{BRS}  D. P. Blecher, Z-J. Ruan, and A. M. Sinclair, {\em   A characterization of operator algebras}, J.\ Functional Analysis {\bf 89} (1990), 288--301.

\bibitem{BWj}  D. P. Blecher and Z. Wang, {\em Jordan operator algebras:\ basic theory,} Mathematische Nachrichten, {\bf 291} (2018), 1629--1654.

\bibitem{BWj2}  D. P. Blecher and Z. Wang, {\em Jordan operator algebras revisited,}
 Mathematische Nachrichten {\bf 292} (2019), 2129--2136.
 
\bibitem{BWinv}  D. P. Blecher and Z. Wang, {\em Involutive operator algebras,} Positivity {\bf 24} (2020), 13--53.  

\bibitem{Rod}  M. Cabrera Garc\'ia and A.  Rodr\'iguez Palacios, {\em 
	Non-associative normed algebras. Vol. 1,}
The Vidav-Palmer and Gelfand-Naimark theorems. Encyclopedia of Mathematics and its Applications, 154. Cambridge University Press, Cambridge, 2014.
\bibitem{CDRV} C.-H. Chu, T. Dang, B. Russo, and B. Ventura, {\em Surjective isometries of real $C^*$-algebras,} Journal of the London Mathematical Society, {\bf 47}(1993), 97--118.

\bibitem{C} A. Connes, {\em A factor not anti-isomorphic to itself,}  Bull. London Math. Soc., {\bf 7} (1975), 171--174.

\bibitem{ER} E. G. Effros and Z-J. Ruan, {\em Operator spaces,} London Mathematical Society Monographs. New Series, 23. The Clarendon Press, Oxford University Press, New York, 2000. 

\bibitem{Good} K. Goodearl, {\em  Notes on real and complex $C^*$-algebras,}
 Shiva Mathematics Series vol. 5, Shiva Publishing Ltd., Nantwitch, 1982.  
\bibitem{Hamiecds} M. Hamana, {\em Injective envelopes of $C^*$-dynamical systems,} Tohoku Math. J. {\bf 37} (1985),  463--487. 


\bibitem{Hamiods} M. Hamana, {\em Injective envelopes of dynamical systems,} Toyama Math. J. {\bf 34} (2011), 23--86.

\bibitem{HS} H. Hanche-Olsen and E. St{\o}rmer, Jordan operator algebras, Pitman Publishing Inc., London, 1984.


\bibitem{Ing}  L. Ingelstam, {\em Real Banach algebras,} Arkiv f\"or Matematik, {\bf  5} (1964),  239--270.

\bibitem{IKR} J. M.  Isidro, W. Kaup, and A.  Rodr\'iguez-Palacios, {\em On real forms of $JB^*$-triples,} Manuscripta Math. {\bf 86} (1995), 311--335. 

\bibitem{IP} J. M.  Isidro and A.  Rodr\'iguez-Palacios, {\em On the definition of real $W^*$-algebras,} Proc.\ Amer.\ Math.\ Soc.\ {\bf 124}
 (1996), 3407--3410.
 
 \bibitem{Kad}   R. V. Kadison, {\em  Isometries of operator algebras,} Ann.\ of Math.\ {\bf 54} (1951), 325--338.
\bibitem{Li}  B. Li, {\em Real operator algebras, } World Scientific, River Edge, N.J., 2003.

\bibitem{M}  J. D. Maitland Wright, {\em Jordan $C^*$-algebras,}   Michigan Math. J. {\bf 24} (1977), 291--302.  

\bibitem{Mos} M. S. Moslehian, G. A. Mu\~noz-Fern\'andez, A. M. Peralta, and J. B. Seoane-Sep\'ulveda, {\em Similarities and differences between real and complex Banach spaces: an overview and recent developments,} preprint (2021),  arXiv:2107.03740 

\bibitem{Pa} T. Palmer, Real  $C^*$-algebras, Pacific J. Math,  {\bf 35} (1970), 195--204.


\bibitem{Pau}   V. I. Paulsen, {\em Completely bounded maps and operator
algebras,} Cambridge Studies in Advanced Math., 78, Cambridge
University Press, Cambridge, 2002.

\bibitem{P} G. K. Pedersen, {\em $C^*$-algebras and their automorphism
groups,} Academic Press, London (1979).

\bibitem{Pisbk}  G.\ Pisier, {\em  Introduction to operator space
theory,} London Math.\ Soc.\ Lecture Note Series, 294, Cambridge
University Press, Cambridge, 2003.  

\bibitem{RodP} A.  Rodr\'iguez Palacios,  {\em  Conditional expectations in complete normed complex algebras satisfying the von Neumann inequality,} Revised preprint, 2021, to appear.

\bibitem{Ros} J. Rosenberg,  {\em Structure and application of real  $C^*$-algebras,} Contemporary Mathematics, {\bf 671} (2016), 235--258.

\bibitem{ROnr} Z-J. Ruan, {\em On real operator spaces,} Acta Mathematica Sinica,  {\bf19} (2003), 485--496.

\bibitem{RComp} Z-J. Ruan, {\em Complexifications of real operator spaces,} Illinois Journal of Mathematics, 
{\bf 47} (2003), 1047--1062.

\bibitem{Sharma} S. Sharma, {\em Real operator algebras and real completely isometric theory,} Positivity {\bf 18} (2014),
95--118.

\bibitem{ST}  E. St{\o}rmer, {\em Positive linear maps of operator algebras,} Springer Monographs in Mathematics, Springer-Verlag (2013). 

\bibitem{WTT}   W. Tepsan,   {\em Real operator spaces, real operator algebras, and real Jordan operator algebras,}
 PhD thesis, University of Houston (May, 2020).

\bibitem{Top}   D. Topping, {\em Jordan algebras of self-adjoint operators,} Bull. Amer. Math. Soc., {\bf 71} (1965), 160--164.

\bibitem{ZWthes}  Z. Wang, {\em Theory of Jordan operator algebras and operator $*$-algebras,} PhD thesis, University of Houston, 2019.


\end{thebibliography}
\end{document}